\documentclass{article}
\usepackage{amsmath, amssymb}
\usepackage[all]{xy}
\newtheorem{thm}{Theorem}[section]
\newtheorem{cor}[thm]{Corollary}
\newtheorem{prop}[thm]{Proposition}
\newtheorem{lemma}[thm]{Lemma}
\newenvironment{dfn}{\medskip\refstepcounter{thm}
\noindent{\bf Definition \thesection.\arabic{thm}\ }}{\medskip}
\newenvironment{ex}{\medskip\refstepcounter{thm}
\noindent{\bf Example \thesection.\arabic{thm}\ }}{\medskip}
\newenvironment{proof}[1][,]{\medskip\ifcat,#1
\noindent{{\it Proof}:\ }\else\noindent{\it Proof of #1.\ }\fi}
{\hfill$\square$\medskip}
\newenvironment{remark}[1][Remark]{\begin{trivlist}
\item[\hskip \labelsep {\bfseries #1}]}{\end{trivlist}}
\newenvironment{remarks}[1][Remarks]{\begin{trivlist}
\item[\hskip \labelsep {\bfseries #1}]}{\end{trivlist}}
\newenvironment{note}[1][Note]{\begin{trivlist}
\item[\hskip \labelsep {\bfseries #1}]}{\end{trivlist}}

\newenvironment{ack}[1][Acknowledgements]{\begin{trivlist}
\item[\hskip \labelsep {\bfseries #1}]}{\end{trivlist}}

\def\O{{\mathbb O}}
\def\C{{\mathbb C}}
\def\Im{\mathop{\rm Im}\nolimits} 
\def\Re{\mathop{\rm Re}\nolimits} 
\def\R{{\mathbb R}}
\DeclareMathOperator\SU{SU}
\DeclareMathOperator\GG{G}
\def\bfe{{\mathbf e}}
\def\bff{{\mathbf f}}
\def\g2{{\mathfrak g}_2}
\def\bfx{{\mathbf x}}
\def\bfu{{\mathbf u}}
\def\bfv{{\mathbf v}}
\def\bft{{\mathbf t}}
\def\bfn{{\mathbf n}}
\def\bfb{{\mathbf b}}
\def\bfw{{\mathbf w}}
\def\P{{\mathbb P}}

\DeclareMathOperator\U{U}
\DeclareMathOperator\GL{GL}
\def\gg{{\rm g}}
\def\d{{\rm d}}
\def\w{\wedge}
\def\eq#1{{\rm(\ref{#1})}}
\DeclareMathOperator\SO{SO}
\def\Z{{\mathbb Z}}
\def\eps{{\varepsilon}}
\DeclareMathOperator\Sym{Sym}
\DeclareMathOperator\Skew{Skew}
\DeclareMathOperator\Tr{Tr}
\DeclareMathOperator\AAA{A}
\DeclareMathOperator\SSS{S}
\DeclareMathOperator\diag{diag}
\DeclareMathOperator\GGG{G}

\DeclareMathOperator\vol{vol}
\DeclareMathOperator\Gr{Gr}
\DeclareMathOperator\Id{Id}

\begin{document}

\title{Ruled Lagrangian Submanifolds of the 6-Sphere}
\author{Jason D. Lotay\\ {\normalsize MSRI, Berkeley}}

\date{}

\maketitle

\begin{center}
{\large\bf Abstract}
\end{center}

\noindent This article sets out to serve a dual purpose.  On the one hand, we give an explicit description of the 
Lagrangians in the nearly K\"ahler 6-sphere $\mathcal{S}^6$ which are \emph{ruled} by circles of constant radius using 
`\emph{Weierstrass formulae}'.  On the other, we recognise all 
known examples of these Lagrangians as being ruled by such circles.  Therefore, 
we describe all families of Lagrangians in $\mathcal{S}^6$ whose second fundamental form satisfies
natural pointwise conditions: so-called `\emph{second order families}'.

\section[Introduction]{Introduction}\label{s1}

\footnote{\textsl{2000 Mathematics Subject Classification: 53B20, 53B25.}}The 6-sphere $\mathcal{S}^6$ inherits a \emph{nearly K\"ahler} structure from its natural inclusion in the imaginary octonions.  It is
thus endowed with an almost symplectic structure, given by a nondegenerate 2-form $\omega$ which is not closed.  We may define \emph{Lagrangian} submanifolds of $\mathcal{S}^6$ as 3-dimensional submanifolds on which $\omega$ vanishes.  
Surprisingly, Lagrangians in $\mathcal{S}^6$ are minimal and real analytic.  They are also related to calibrated
 4-dimensional submanifolds of $\R^7$ known as \emph{coassociative} 4-folds and are analogues of \emph{special Lagrangian}
 submanifolds of $\C^3$.


\subsection[Motivation]{Motivation}
 
 Lagrangian submanifolds of $\mathcal{S}^6$ are studied by many authors and several families of explicit examples are known.  
The homogeneous examples are classified in \cite{Mashimo} and the constant curvature examples in \cite{Ejiri}.  Some of 
these examples are given explicit descriptions in \cite{Dillen}.  
 Lagrangians satisfying certain curvature conditions are classified in \cite{Deszcz} and \cite{Dillen2}.  
  Ruled Lagrangian submanifolds in $\mathcal{S}^6$ 
are equivalent to coassociative cones in $\R^7$ ruled by 2-planes, 
which are studied in \cite{Fox2} and by the author in \cite{Lotay2R}.  A special family of ruled Lagrangians is analysed  
in \cite{Vrancken2}.
 
The study of second order families of submanifolds associated with special holonomy was begun by Bryant \cite{Bryant2} 
(for special Lagrangian 3-folds in $\C^3$) and continued in \cite{Fox} and \cite{Ionel}.   
  Bryant \cite[$\S$3]{Bryant2} classifies all special Lagrangian 3-folds whose second fundamental has a pointwise symmetry and
 gives a characterisation of ruled examples.   However, he does not explicitly describe the 
ruled family and remarks that there cannot be a Weierstrass formula for the general ruled special Lagrangian 3-fold.  
Ruled special Lagrangians in $\C^3$ are also studied in \cite{JoyceR}, and the analogous situation in $\R^7$, for 
 calibrated submanifolds called \emph{associative} 3-folds, 
is researched by the author in \cite{LotayR}.  

The full classification of Lagrangians in $\mathcal{S}^6$ whose second fundamental form has a nontrivial stabilizer is
 given in \cite{Vrancken}, though the exposition is rather brief.  
Here we give more detail and a new perspective on this survey.  

\subsection[Summary]{Summary}

The culmination of the work in this article is paraphrased in Theorems \ref{thm1} and \ref{thm2} below.  
  Recall that the 
6-sphere is endowed with an almost 
complex structure and thus we can define \emph{pseudoholomorphic curves} $\Sigma$ in $\mathcal{S}^6$.  If $\Sigma$ is 
non-totally geodesic we can define its first and second normal bundles $N_1\Sigma$ and $N_2\Sigma$ respectively, as well 
as its full normal bundle $N\Sigma$.  
Given any pseudoholomorphic curve $\Sigma$ we can construct a \emph{tube} 
in $N\Sigma$ which is a 3-dimensional submanifold of $\mathcal{S}^6$ ruled by circles of constant radius.  
A special family of pseudoholomorphic curves are those with \emph{null-torsion}.  
 For formal definitions we refer the reader to $\S$\ref{s2}-$\S$\ref{s4}.  

We shall denote a totally 
geodesic $n$-sphere in $\mathcal{S}^6$ by $\mathcal{S}^n$, for $n<6$.  We say that a submanifold of $\mathcal{S}^6$ is 
\emph{linearly full} if it is not contained in any $\mathcal{S}^n$ in $\mathcal{S}^6$ for $n<6$.  

\begin{thm}\label{thm1} Let $L$ be a Lagrangian in $\mathcal{S}^6$ and let $h_L$ be its second fundamental form.
\begin{itemize}
\item[\emph{(a)}] $L$ is not linearly full if and only if $L$ is the Hopf lift in $\mathcal{S}^5$ of a holomorphic
 curve in $\C\P^2$ as in Example \ref{s4ex3}.  Moreover, $L$ is 
ruled by geodesic circles and $h_L$ has a pointwise symmetry.
\item[\emph{(b)}]
$L$ is linearly full and $h_L$ admits a pointwise symmetry if and only if 
$L$ is locally either a tube in $N_2\Sigma$ ruled by geodesic circles over 
a pseudoholomorphic curve $\Sigma\neq\mathcal{S}^2$,  
or a tube ruled by circles of radius $\frac{2}{3}$ over a null-torsion pseudoholomorphic curve $\Sigma$, which is in $N_2\Sigma$ if 
$\Sigma\neq\mathcal{S}^2$. 
\end{itemize}
\end{thm}

\begin{note}
We can define a Lagrangian tube about $\mathcal{S}^2$ ruled by circles of radius $\frac{2}{3}$, though we can no longer distinguish the first and second normal bundles.
\end{note}

\begin{thm}\label{thm2} Let $L$ be a linearly full Lagrangian in $\mathcal{S}^6$ and let 
$h_L$ denote its second fundamental form. 
\begin{itemize}
\item[\emph{(a)}]  Suppose $h_L$ does not admit a pointwise symmetry. Then
$L$ is ruled by geodesic circles if and only if $L$ is locally a tube  
about a pseudoholomorphic curve $\Sigma\neq\mathcal{S}^2$, 
defined using a holomorphic curve in 
a $\C\P^1$-bundle over $\Sigma$ as in Example \ref{s7ex1}.
\item[\emph{(b)}]
$L$ is ruled by non-geodesic circles of constant radius if and only if $L$ is locally a tube about a null-torsion pseudoholomorphic curve $\Sigma$ (in $N_2\Sigma$ when 
$\Sigma\neq\mathcal{S}^2$) and the ruling circles have radius $\frac{2}{3}$.  Moreover, $h_L$ admits a pointwise 
symmetry.
\end{itemize}
\end{thm}

We begin in $\S$\ref{s2} by defining the nearly K\"ahler structure on the 6-sphere and the submanifolds 
associated with this structure.  In $\S$\ref{s3} we introduce the \emph{fundamental cubic} of a Lagrangian in $\mathcal{S}^6$,  
defined by its second fundamental form, and
 derive two presentations of the structure equations of $\GGG_2$ which are adapted to the Lagrangian and pseudoholomorphic geometries. 
 In $\S$\ref{s4} we define ruled Lagrangians and characterise them using CR-holomorphic curves in the 
 space of geodesic circles in $\mathcal{S}^6$.  We also define tubes about surfaces in $\mathcal{S}^6$, 
which provide examples of ruled 3-dimensional submanifolds of 
$\mathcal{S}^6$. 

In $\S$\ref{s5} we start by studying the possible pointwise symmetries of the fundamental cubic.
We then explicitly describe the Riemann curvature tensors associated with 
 the induced metrics on Lagrangian submanifolds of $\mathcal{S}^6$.  In $\S$\ref{s6} 
 we classify the Lagrangians in $\mathcal{S}^6$ whose fundamental cubic admits a pointwise symmetry. Furthermore, we
recognise these examples as Hopf lifts of holomorphic curves or locally tubes about pseudoholomorphic curves. 
 Theorem 
\ref{thm1} follows from these considerations.

In $\S$\ref{s7} we give an explicit description of the general ruled Lagrangian.  We also classify
 the Lagrangians that are ruled by geodesic circles.  Finally, in $\S$\ref{s8}, we show that the Lagrangians 
which are ruled by non-geodesic circles of constant radius have already been described in $\S$\ref{s6}.  These 
results prove Theorem \ref{thm2}.
 
\begin{note}
In this article, we shall occasionally use the theory of exterior differential systems (EDS).  We  
will adopt the standard notion of local dependence of the solutions to an involutive system on $m$ functions of $n$ variables 
when the last non-zero Cartan character is $s_n=m$.  The author first came to understand the material discussed in this article 
through EDS analysis and thought it useful to include some of these considerations.  
However, it is not essential for the reader to be familiar with EDS in order to understand the paper. 
\end{note}

\section[Submanifolds of the nearly Kaehler 6-sphere]{Submanifolds of the nearly K\"ahler 6-sphere}\label{s2}

We shall view the 6-sphere as a \emph{nearly K\"ahler} 6-manifold, which we now define.

\begin{dfn}\label{s2dfn1}
Let $(M,g,J,\omega)$ be an almost Hermitian 6-manifold: that is, $g$ is a Riemannian metric on the almost complex 
6-manifold $M$, $J$ is the almost complex structure preserved by $g$, and $\omega$ is the associated (nondegenerate) 
Hermitian $(1,1)$-form.   The manifold $M$ is \emph{nearly K\"ahler} if there exists a nowhere vanishing $(3,0)$-form 
$\Omega$ on $M$ and a non-zero real constant $\lambda$ such that 
\begin{equation}\label{nkahler}
\d\omega=3\lambda\Re\Omega\quad\text{and}\quad \d\Im\Omega=-2\lambda\,\omega\w\omega.
\end{equation}
Equivalently, the $(2,1)$-tensor $G(X,Y)=\nabla_X(J)Y$, for vector fields $X,Y$ on $M$, is skew-symmetric and non-zero, 
where $\nabla$ is the Levi--Civita connection of $g$.  
\end{dfn}

\noindent Nearly K\"ahler 6-manifolds have structure group $\SU(3)$.  Some authors allow the possibility 
that $\lambda=0$ with $\d\Re\Omega=0$, which we will see in Definition \ref{CYSLdfn} is equivalent to including \emph{Calabi--Yau} 3-folds.  Excluding this case means that 
we have defined what other authors call \emph{strictly} nearly K\"ahler 6-manifolds.

It is well-known that a 7-dimensional cone has holonomy $\GGG_2$ if and only if its link is nearly K\"ahler.  Thus, 
$\GGG_2$ and 6-dimensional nearly K\"ahler geometry are intimately intertwined.  In particular, the nearly K\"ahler structure on $\mathcal{S}^6$ 
is induced by the $\GGG_2$ structure on the imaginary octonions $\Im\O$, recalling that $\GGG_2$ is the automorphism group of the 
cross product algebra of $\Im\O$.  It is therefore illuminating to briefly discuss distinguished submanifolds of $\Im\O$.

\begin{dfn}\label{asscoass}
Let $g_0$ be the Euclidean metric on $\Im\O$ and let $\times$ denote the octonionic cross product.  
We can define $\GGG_2$ as the stabilizer in $\GL(\Im\O)$ 
of a 3-form $\varphi_0$ on $\Im\O$ which encodes the octonionic cross product via 
\begin{equation}\label{phitimes}
\varphi_0(u,v,w)=g_0(u\times v,w).
\end{equation}
  This 3-form is closed and coclosed and, by \cite[Theorems IV.1.4 \& IV.1.16]{HarLaw},
$\varphi_0$ and its Hodge dual $*\varphi_0$ are \emph{calibrations}; that is, they are closed forms which satisfy 
$\varphi_0|_U\leq\vol_U$ and $*\varphi_0|_V\leq\vol_V$, where $U$ and $V$ are oriented tangent $3$- and $4$-planes respectively. 
 Submanifolds calibrated with respect to $\varphi_0$
and $*\varphi_0$, i.e.~those on which they restrict to be volume forms, are
called \emph{associative} 3-folds and \emph{coassociative} 4-folds respectively.  We can, by the work in 
\cite[$\S$IV]{HarLaw}, equivalently define coassociative 4-folds as the oriented 4-dimensional submanifolds of $\Im\O$ on which 
$\varphi_0$ vanishes (up to a choice of orientation).
\end{dfn}

\noindent For more details on calibrated geometry and these submanifolds, we recommend \cite{Harvey} and \cite{Joyce}.

We now define the nearly K\"ahler structure on $\mathcal{S}^6$.

\begin{dfn}\label{6spheredfn}
Embedding $\mathcal{S}^6\hookrightarrow\Im\O$ as the unit sphere, we endow $\mathcal{S}^6$ with the round metric $g$ and 
 identify $T_p\mathcal{S}^6$ with the 6-plane in 
$\Im\O$ orthogonal to $p$.  Therefore, we can define a map $J_p:T_p\mathcal{S}^6\rightarrow T_p\mathcal{S}^6$ via 
left multiplication: $J_p(u)=p\times u$.  Elementary octonionic algebra shows that $J_p^2=-1$, so we have an almost complex 
structure $J$ on $\mathcal{S}^6$.

Notice that, on $\Im\O$, we can write 
\begin{equation}\label{phi}
\varphi_0=r^2\d r\w\omega+r^3\Upsilon
\end{equation}
for some 2-form $\omega$ and 3-form $\Upsilon$ on $\mathcal{S}^6$, where $r$ is the radial coordinate on $\Im\O$.  Identifying 
 $p\in\mathcal{S}^6$ with a unit radial vector, we can use \eq{phitimes} and the definition of $J$ to show that
 $\omega$ is the nondegenerate $(1,1)$-form associated with $J$.

Let $\Omega=\Upsilon+i\!*\!\Upsilon$, where $*$ is the Hodge star on $\mathcal{S}^6$.  Notice that, on $\Im\O$, the Hodge dual of $\varphi_0$ can be
written: 
\begin{equation}\label{starphi}
*\varphi_0=\frac{1}{2}\,r^4\omega\w\omega-r^3\d r\w\! *\Upsilon.
\end{equation}
Therefore, as $\varphi_0$ and $*\varphi_0$ are closed,
$\d\omega=3\Upsilon$ and $\d\!*\!\Upsilon=-2\omega\w\omega$.  Finally, we see that $\Omega$ is a nowhere vanishing $(3,0)$-form and 
that $\omega$ and $\Omega$ satisfy \eq{nkahler} for $\lambda=1$.  Hence, $\mathcal{S}^6$ is a nearly K\"ahler 6-manifold.
\end{dfn}

\begin{remarks}
The almost complex structure $J$ on $\mathcal{S}^6$ is not integrable.  Moreover, the 2-form $\omega$ is clearly 
not closed, but it does satisfy $\omega\w\d\omega=0$.
\end{remarks}

Having defined the nearly K\"ahler structure on the 6-sphere, we can present the class of submanifolds we wish to study.

\begin{dfn}
A 3-dimensional submanifold $L$ of $\mathcal{S}^6$ is \emph{Lagrangian} if $\omega|_L\equiv 0$.  Equivalently, 
$J_p(T_pL)=N_pL$ for all $p\in L$. 
\end{dfn}

\noindent Here we have generalised the notion of Lagrangian submanifold, usually reserved for symplectic manifolds, to the 
 \emph{almost symplectic} 6-sphere.  However, since $\mathcal{S}^6$ has a nearly K\"ahler structure,
Lagrangians in $\mathcal{S}^6$ have more properties than one would expect from the general almost symplectic case.

\begin{remark}
Many authors refer to Lagrangians in $\mathcal{S}^6$ as \emph{totally real} 3-dimensional submanifolds of $\mathcal{S}^6$.
\end{remark}

We now show the connection between Lagrangian geometry in $\mathcal{S}^6$ and $\GGG_2$ geometry in $\Im\O$, and give some of its 
consequences.

\begin{prop}\label{coasscones} 
A 4-dimensional cone in $\Im\O$ is coassociative if and only if its link in $\mathcal{S}^6$ is Lagrangian.
\end{prop}

\begin{proof}
Recall that a 4-dimensional cone $C$ in $\Im\O$ is coassociative if and only if $\varphi_0|_C\equiv 0$.  From  
\eq{phi}, we see that $\varphi_0|_C\equiv 0$ if and only if $\omega$ and $\Upsilon$ vanish on its link $L$.  However, $\Upsilon$ 
is a non-zero multiple of $\d\omega$, so $\Upsilon|_L\equiv0$ if $\omega|_L\equiv 0$.
\end{proof}

Since coassociative 4-folds are minimal \cite[Theorem II.4.2]{HarLaw} (in fact, 
volume minimizing in their homology class \cite[Theorem 7.5]{Harvey}) and are real analytic whereever 
they are nonsingular \cite[Theorem 12.1.5]{Joyce} we 
deduce the following.

\begin{cor}
Lagrangians in $\mathcal{S}^6$ are minimal and real analytic away from their singularities.
\end{cor}

\begin{remark}
The minimality of Lagrangians in $\mathcal{S}^6$ is observed in \cite[Theorem 1]{Ejiri}.
\end{remark}

We can also perhaps explain these properties of Lagrangians in $\mathcal{S}^6$ by considering them 
as ``\emph{special Lagrangian}'' submanifolds of $\mathcal{S}^6$.  We take the time now to remind the reader of the 
definition of special Lagrangian $m$-folds.

\begin{dfn}\label{CYSLdfn}
Let $(M,g,J,\omega)$ be a (complex) $m$-dimensional K\"ahler manifold, where $g$ is the metric, $J$ is the complex 
structure and $\omega$ is the K\"ahler form.  We say that $M$ is a \emph{Calabi--Yau} $m$-fold if $M$ is endowed 
with a nowhere vanishing $(m,0)$-form $\Omega$ such that $\d\Re\Omega=\d\Im\Omega=0$.

A real oriented $m$-dimensional submanifold $L$ of a Calabi--Yau $m$-fold $M$ is \emph{special Lagrangian} (with phase $e^{i\theta}$) 
if $\omega|_L\equiv 0$ and $\Im e^{-i\theta}\Omega|_L\equiv 0$.   Equivalently, $\Re e^{-i\theta}\Omega$ is a calibration 
on $M$ and special Lagrangian $m$-folds with phase $e^{i\theta}$ are its calibrated submanifolds; 
i.e.~$L$ satisfies $\Re e^{-i\theta}\Omega|_L=\vol_L$.
\end{dfn}

If $L$ is a 3-dimensional submanifold of $\mathcal{S}^6$ satisfying $\omega|_L\equiv0$, 
then $\Upsilon=\Re\Omega$ also vanishes on $L$ as $3\Upsilon=\d\omega$.  Thus, 
$L$ is Lagrangian in $\mathcal{S}^6$ if and only if $\omega|_L=\Re\Omega|_L\equiv0$.  Notice from \eq{starphi} that,
for any oriented tangent 3-plane $V$, $-\!*\!\Upsilon|_V\leq\vol_V$, since $*\varphi_0$ is a calibration.  
Therefore, $-\!*\!\Upsilon=-\Im\Omega$ satisfies the condition to be a calibration on $\mathcal{S}^6$, except that it is not closed.  
Moreover, using \eq{starphi} in 
conjunction with Proposition \ref{coasscones}, we observe that $-\!*\!\Upsilon|_L=-\Im\Omega|_L=\vol_L$ for a Lagrangian $L$ in $\mathcal{S}^6$.  
Therefore, Lagrangians in $\mathcal{S}^6$ are, in this sense, ``special Lagrangian'' (with phase $-i$) 
with respect to the nearly K\"ahler structure. This leads a few authors to call
 Lagrangian submanifolds of the 6-sphere ``special Lagrangian'', though we are disinclined to join them in this notation.  
That said, 
 we can continue the special Lagrangian analogy as follows.
 
\begin{prop}\label{extend}
Let $P$ be a real analytic 2-dimensional submanifold of $\mathcal{S}^6$ such that $\omega|_P\equiv 0$.  Locally 
there exists a Lagrangian submanifold $L$ of $\mathcal{S}^6$ containing $P$.  Moreover, $L$ is locally unique.
\end{prop}

\begin{proof}
 Let $\omega_0$ and $\Omega_0$ be the K\"ahler and holomorphic volume forms associated with the standard Calabi--Yau structure 
 on $\C^3$.  
Using the Cartan--K\"ahler Theorem one can prove the analogous statement of the proposition \cite[Theorem III.5.5]{HarLaw}
for special Lagrangians (with phase $-i$) in $\C^3$ by 
considering the exterior differential system on $\C^3$ with differential ideal generated by $\omega_0$ 
and $\Re\Omega_0$.  By the comments above, for Lagrangians in $\mathcal{S}^6$ we are lead to consider
 the EDS with ideal generated by $\omega$ and $\Upsilon$.  Algebraically 
the ideals have the same properties, since they are both algebraically generated by a 2-form and a 3-form (in the 
special Lagrangian case because they are both closed, and in the Lagrangian case because $3\Upsilon=\d\omega$).  Since the proof using the Cartan--K\"ahler Theorem only relies on the
 algebra of the EDS, the proposition follows.
\end{proof}

By the proof of Proposition \ref{extend}, special Lagrangians in $\C^3$ and Lagrangians in $\mathcal{S}^6$ 
have the same `local existence' properties: they both depend 
locally on 2 functions of 2 variables.  In contrast, Lagrangians in symplectic 6-manifolds, by 
Darboux's Theorem, depend locally on an arbitrary function of 3 variables 
 (and there is no need for real analyticity).
 
\medskip

Just as there is interplay between the complex and symplectic geometry of a K\"ahler manifold, there are connections between  
the almost complex and almost symplectic geometries of the nearly K\"ahler 6-sphere.  This leads us to define another 
distinguished class of submanifolds of $\mathcal{S}^6$. 

\begin{dfn}
A 2-dimensional submanifold $\Sigma$ of $\mathcal{S}^6$ is a \emph{pseudoholomorphic curve} if $\omega|_{\Sigma}=\vol_{\Sigma}$.  
Equivalently, $J_p(T_p\Sigma)=T_p\Sigma$ for all $p\in \Sigma$.
\end{dfn}

\begin{remark}
There are no almost complex 4-folds in $\mathcal{S}^6$ \cite[Proposition 4.1]{Bryant1}.
\end{remark}

From \eq{phi}, we observe that a 3-dimensional cone in $\Im\O$ is associative if and only if its
 link in $\mathcal{S}^6$ is a pseudoholomorphic curve.   
We may readily deduce some of the well-known properties of pseudoholomorphic curves.  
 First, they are minimal and real analytic away from their 
 singularities.  Second, every real analytic curve in $\mathcal{S}^6$ can be locally extended in a locally unique way to a 
 pseudoholomorphic curve.  Hence, pseudoholomorphic curves in $\mathcal{S}^6$ depend locally on 
 4 functions of 1 variable. 


\begin{remarks}
We can define Lagrangian and almost complex submanifolds in any nearly K\"ahler 6-manifold, in the obvious manner.  Since the 
results of this section on submanifolds of $\mathcal{S}^6$ 
follow directly from the definition of a nearly K\"ahler structure in 6 dimensions, and its relationship 
with $\GG_2$ geometry, they can all be naturally generalised for the corresponding submanifolds of arbitrary nearly 
K\"ahler 6-manifolds.  However, most of the later results in this article rely heavily on the particular geometry of $\mathcal{S}^6$ 
and so will not immediately generalise.
\end{remarks}

\section[The structure equations and the fundamental cubic]{The structure equations and the fundamental cubic}\label{s3}

In this section we provide the formulae that we require for our calculations later.  In particular, we 
view the frame bundle over $\mathcal{S}^6$ as $\GGG_2$, since $\mathcal{S}^6\cong\GGG_2/\SU(3)$, and thus
 give two presentations of the structure equations of $\GGG_2$ adapted to the Lagrangian and pseudoholomorphic 
curve scenarios.  Along the way we introduce the fundamental cubic of a Lagrangian submanifold, 
which is a useful means of encoding the second fundamental form.

\medskip

Since the nearly K\"ahler structure on $\mathcal{S}^6$ is defined using the octonions $\O$, it is useful for reference 
to have the multiplication table for $\O$.  Let $\O$ be 
spanned by $1$ and $\{\eps_1,\ldots,\eps_7\}$ satisfying the multiplication law below.  
\begin{equation*}
\begin{gathered}
\begin{array}{rrrrrrrrrr}
  & \vline & 1  &  \eps_1 &  \eps_2 &  \eps_3 &  \eps_4 &  \eps_5 &  \eps_6 &  \eps_7 \\
\hline
  1 & \vline &  1  &  \eps_1 &  \eps_2 &  \eps_3 &  \eps_4 &  \eps_5 &  \eps_6 &  \eps_7 \\
\eps_1 & \vline & \eps_1 &  -1  &  \eps_3 & -\eps_2 &  \eps_5 & -\eps_4 &  \eps_7 & -\eps_6 \\
\eps_2 & \vline & \eps_2 & -\eps_3 &  -1  &  \eps_1 &  \eps_6 & -\eps_7 & -\eps_4 &  \eps_5 \\
\eps_3 & \vline & \eps_3 &  \eps_2 & -\eps_1 &  -1  & -\eps_7 & -\eps_6 &  \eps_5 &  \eps_4 \\
\eps_4 & \vline & \eps_4 & -\eps_5 & -\eps_6 &  \eps_7 &  -1  &  \eps_1 &  \eps_2 & -\eps_3 \\
\eps_5 & \vline & \eps_5 &  \eps_4 &  \eps_7 &  \eps_6 & -\eps_1 &  -1  & -\eps_3 & -\eps_2 \\
\eps_6 & \vline & \eps_6 & -\eps_7 &  \eps_4 & -\eps_5 & -\eps_2 &  \eps_3 &  -1  &  \eps_1 \\
\eps_7 & \vline & \eps_7 &  \eps_6 & -\eps_5 & -\eps_4 &  \eps_3 &  \eps_2 & -\eps_1 &  -1
\end{array}
\end{gathered}
\end{equation*}
Then $\Im\O$, the imaginary octonions, is spanned by the $\eps_j$.  We shall denote the cross product on $\Im\O$ 
by $\times$, as is standard practice.

\begin{note}
Some authors use a basis for $\Im\O$ which has the opposite orientation: 
the difference in our formulae can be accounted for by a change of sign of $\eps_7$. 
\end{note}

\subsection[Lagrangian submanifolds]{Lagrangian submanifolds}\label{s3subs1}


Let $\bfx:L\rightarrow\mathcal{S}^6$ be a Lagrangian immersion and let $g_L$ be the induced metric on $L$. 
At each point $p$ in $L$, let $\{\bfe_1(p),\bfe_2(p),\bfe_3(p)\}$ be an orthonormal
 basis for $T_pL$ and let $\{2\omega_1(p),2\omega_2(p),2\omega_3(p)\}$ define the dual orthonormal coframe.     Notice
that $\{J\bfe_1(p),J\bfe_2(p),J\bfe_3(p)\}$ defines an orthonormal basis for $N_pL$.
%
 We will implicitly identify $L$ with its image in $\mathcal{S}^6$ and identify the tangent and 
normal vectors at $p\in L$ with their push-forwards in $T_p\mathcal{S}^6\cong\langle p\rangle^{\perp}\subseteq\Im\O$. 


\begin{note}
We define $2\omega_j$ to be the dual 1-form to the vector $\bfe_j$ to make our 
 structure equations below look neater.
\end{note}

We now introduce the fundamental cubic of a Lagrangian submanifold.

\begin{dfn}\label{funcubic}
Adopting summation notation, we can locally write the second fundamental form of $L$ as 
$$h_L=4 h_{ijk}J\bfe_i\otimes\omega_j\omega_k$$
for some totally symmetric tensor of functions $h_{ijk}$, which satisfies $h_{iik}=0$ as $L$ is minimal. 
 We may therefore define the \emph{fundamental cubic} $C_L$ of $L$ as
$$ C_L=8 h_{ijk}\omega_i\omega_j\omega_k. $$
Thus, $C_L$ is a section of $S^3T^*L$ which encodes the second fundamental form of $L$.  
\end{dfn}

\noindent We can realise $C_L$ pointwise as a homogeneous harmonic cubic on $\R^3$.  
This picture will be useful from 
an algebraic standpoint.

\medskip

We now derive an expression for the structure equations of $\GGG_2$ best suited to Lagrangian geometry.
 The Lie algebra of $\GGG_2$, $\g2$, has the following matrix presentation:
\begin{align*}
\mathfrak{g}_2 &=\left\{\left(\begin{array}{ccc} 0 & -2\omega^{\rm T} & -2\eta^{\rm T} \\
2\omega & \alpha+[\omega] & -\beta-[\eta] \\
2\eta & \beta-[\eta] & \alpha - [\omega] \end{array}\right)\,:\,\omega,\eta\in \text{M}_{3\times 1}(\mathbb{R}),\right.\\
&\qquad\qquad\qquad\qquad\qquad\qquad\qquad\qquad\qquad\,\left.
\alpha\in\Skew_{3}(\mathbb{R}),\, \beta\in\Sym_{3}^0(\mathbb{R})\right\},
\end{align*}
where $\Sym_3^0(\mathbb{R})$ is the space of traceless symmetric $3\times 3$ real matrices and
$$[(x\;y\;z)^{\rm T}]=\left(\begin{array}{ccc} 0 & z & -y \\
-z & 0 & x \\
y & -x & 0\end{array}\right).$$ 
Let $\gg:\GGG_2\rightarrow\GL(7,\R)$ be the map taking $\GGG_2$ to the identity component of the Lie subgroup of 
$\GL(7,\R)$ with Lie algebra $\g2$.  Write $\gg=(\bfx\; \bfe\; \bfe^{\perp})$, where $\bfe=(\bfe_1\;\bfe_2\;\bfe_3)$ and 
$\bfe^{\perp}=(\bfe_1^{\perp}\;\bfe_2^{\perp}\;\bfe_3^{\perp})$ are in $\text{M}_{7\times 3}(\R)$.  
Since the Maurer--Cartan form $\phi=\gg^{-1}\d\gg$ takes values in $\g2$, it can be written as
$$\phi=\left(\begin{array}{ccc} 0 & -2\omega^{\rm T} & -2\eta^{\rm T} \\
2\omega & \alpha+[\omega] & -\beta-[\eta] \\
2\eta & \beta-[\eta] & \alpha - [\omega] \end{array}\right)$$
for some appropriate matrices of 1-forms $\omega, \eta, \alpha, \beta$.  

We can adapt frames on $L$ so that $\bfx$ is identified with a point in $L$, and $\bfe$ and $2\omega$ define an orthonormal frame
and coframe for $L$.  Thus, we can set $\bfe^{\perp}=J\bfe$ and see that $\eta$ vanishes on $L$.  
From this adaptation, we recognise $\alpha+[\omega]$ as
 the connection 1-form for the Levi--Civita connection $\nabla^L$ of the metric $g_L$.  

From $\d\gg=\gg\phi$ and the Maurer-Cartan equation $\d\phi+\phi\w\phi=0$,
 we immediately derive the \emph{first} and \emph{second} \emph{structure equations} of $\GGG_2$.

\begin{prop}
\label{s2prop1} 
The first structure of equations of $\GGG_2$ can be written:
\begin{align*}
\d\bfx&=2\bfe\omega+2\bfe^{\perp}\eta;\\
\d\bfe&=-2\bfx\omega^{\rm T}+\bfe(\alpha+[\omega])+\bfe^{\perp}(\beta-[\eta]);\\
\d\bfe^{\perp}&=-2\bfx\eta^{\rm T}-\bfe(\beta+[\eta])+\bfe^{\perp}(\alpha-[\omega]).
\end{align*}
On the adapted frame bundle of $L$, these equations become:
\begin{subequations}\label{firstL}
\begin{align}
\d\bfx&=2\bfe\omega;\\
\d\bfe&=-2\bfx\omega^{\rm T}+\bfe(\alpha+[\omega])+J\bfe\beta;\\
\d J \bfe&=-\bfe\beta+J\bfe(\alpha-[\omega]).
\end{align}
\end{subequations}
\end{prop}


\begin{prop}
\label{s2prop2} 
The second structure equations of $\GGG_2$ are:
\begin{align*}
\d\omega=&-(\alpha+[\omega])\w\omega+(\beta+[\eta])\w\eta;\\
\d\eta=&-(\beta-[\eta])\w\omega-(\alpha-[\omega])\w\eta;\\
\d\alpha=&-\alpha\w\alpha+\beta\w\beta+3\omega\w\omega^{\rm T}+3\eta\w\eta^{\rm T};\\
\d\beta=&-\alpha\w\beta-\beta\w\alpha-2\omega\w\eta^{\rm T}+2\eta\w\omega^{\rm T}-[\omega]\w[\eta]+[\eta]\w[\omega].
\end{align*}
On the adapted frame bundle of $L$ there exists a fully symmetric tensor of functions $h=h_{ijk}$ such that 
the structure equations become: 
\begin{subequations}\label{secondL}
\begin{align}
\d\omega&=-(\alpha+[\omega])\w\omega;
\label{secondLa}\\
\beta&=2h\omega;
\label{secondLb}\\
\d\alpha&=-\alpha\w\alpha+\beta\w\beta+3\omega\w\omega^{\rm T};
\label{secondLc}\\
\d\beta&=-\alpha\w\beta-\beta\w\alpha.\label{secondLd}
\end{align}
\end{subequations}
Therefore, on the adapted frame bundle of $L$,
\begin{subequations}
\begin{align}
\d(\alpha+[\omega])+(\alpha+[\omega])\w(\alpha+[\omega])&=4\big(h\omega\w h\omega+\omega\w\omega^{\rm T}\big)\quad\text{and}
\label{Ga}\\
\d h+\big(\big(h\alpha+\textstyle\frac{1}{2}h[\omega]\big)\big)&=H\omega
\label{Co}
\end{align}
\end{subequations}
for some fully symmetric tensor of functions $H=H_{ijkl}$, where $\big(\big(\,\,\big)\big)$ indicates symmetrisation over the indices:
 i.e.~in summation notation,
$$\big(\big(h\alpha\big)\big)_{ijk}=h_{lij}\alpha_{kl}+h_{ljk}\alpha_{il}+h_{lki}\alpha_{jl}.$$
\end{prop}

Here, $h_{ijk}$ defines the fundamental cubic $C_L$ of $L$ as in Definition \ref{funcubic}. 
 Recalling that the connection 1-form of the metric $g_L$ is $\alpha+[\omega]$, 
 the 
 equations \eq{Ga} and \eq{Co} can be interpreted as Gauss and
Codazzi-like equations.  Explicitly, if $R_{ijkl}=\text{Riem}(g_L)$, \eq{Ga} and \eq{Co} are equivalent to:
\begin{align*}
R_{ijkl}=\sum_q (h_{ikq}h_{jlq}-h_{ilq}h_{jkq})+\delta_{ik}\delta_{jl}-\delta_{il}\delta_{jk}\!\quad\!\text{and}\!\quad\!
\nabla^L C_L\in \Gamma(S^4T^*L).
\end{align*}
These conditions are necessary and sufficient for $(L,g_L)$ to be isometrically embedded as a 
Lagrangian submanifold of $\mathcal{S}^6$ with fundamental cubic $C_L$.  



\subsection[Pseudoholomorphic curves]{Pseudoholomorphic curves}\label{pholo}

Recall that, since we have an almost complex structure $J$ on $\mathcal{S}^6$ we can define, at each $p\in\mathcal{S}^6$, 
$T^{1,0}_p\mathcal{S}^6=\{\bfv\in T_p\mathcal{S}^6\otimes_\R\C\,:\,J_p\bfv=i\bfv\}$.  A unitary basis for 
$T_p\mathcal{S}^6$ is then really a complex basis for $T^{1,0}_p\mathcal{S}^6$ whose elements are orthogonal and of length 
$\frac{1}{\sqrt{2}}$.

Let $\bfu:\Sigma\rightarrow\mathcal{S}^6$ be a pseudoholomorphic curve and let $g_{\Sigma}$ be the induced metric on $\Sigma$.
 At each point $p\in\Sigma$, we can decompose $T^{1,0}_p\mathcal{S}^6|_\Sigma=T_p^{1,0}\Sigma\oplus N_p\Sigma$.  Let $\bff_1(p)$ span the holomorphic tangent space $T^{1,0}_p\Sigma$ 
and let $\theta_1(p)$ be the dual 1-form.  
%
 Let $N_p\Sigma=\langle\bff_2(p),\bff_3(p)\rangle_{\C}$ 
and let $\{\theta_2(p),\theta_3(p)\}$ be the obvious dual 1-forms.  
We can further choose $\{\bff_1(p),\bff_2(p),\bff_3(p)\}$ to be a unitary frame for $T_p\mathcal{S}^6$ at each $p\in\Sigma$, 
and thus the $\theta_i(p)$ define a dual unitary coframe.   

Recall that the second fundamental form $h_\Sigma$ of $\Sigma$, when evaluated on tangent vectors, takes values in the normal 
bundle of $\Sigma$ in $\mathcal{S}^6$.  Thus, at every non-totally geodesic point $p\in\Sigma$, we can define $N_1\Sigma|_p$ 
to be the subspace of $N_p\Sigma$ in which $h_\Sigma$ takes values when evaluated on elements of $T^{1,0}_p\Sigma$, and 
further define $N_2\Sigma|_p$ to be the subspace of $N_p\Sigma$ orthogonal to $N_1\Sigma|_p$.  Thus, if $\Sigma$ 
is not totally geodesic, $N\Sigma$ decomposes into 
holomorphic line bundles $N_1\Sigma$ and $N_2\Sigma$, called the \emph{first} and \emph{second} normal bundles.  In this case, 
 one could adapt frames so that $\bff_2$ spans $N_1\Sigma$ and 
$\bff_3$ spans $N_2\Sigma$.  
However, we refrain from making this choice in general.

\medskip

We shall need the following cross products, which can be calculated by taking explicit imaginary octonionic 
representatives for $\bfu$, $\bff_1$, $\bff_2$ and $\bff_3$:
\begin{subequations}\label{times}
\begin{gather}
\bff_1\times\bar{\bff}_1=\bff_2\times\bar{\bff}_2=\bff_3\times\bar{\bff}_3=\frac{i}{2}\bfu;\\
\bff_2\times\bff_3=\bar{\bff}_1;\qquad\bff_3\times\bff_1=\bar{\bff}_2;\qquad\bff_1\times\bff_2=\bar{\bff}_3.
\end{gather}
\end{subequations}

We can explicitly define a unitary 
framing for $T\mathcal{S}^6|_{\Sigma}$ in a neighbourhood $U\subseteq\Sigma$ of each non-totally geodesic point $p\in\Sigma$
 as follows.  Let $\bft_1$ be a unit tangent vector on $U$ and let $\bft_2=J^{\Sigma}\bft_1$, 
where $J^{\Sigma}$ is the almost complex structure on $\Sigma$. 
 Notice that $|h|=\|h_{\Sigma}(\bft,\bft)\|$ is independent of the choice of unit tangent vector $\bft$ on $U$.
  Therefore, identifying the tangent vectors in $\Sigma$ with their push-forwards in $T\mathcal{S}^6|_{\Sigma}$, define:
\begin{subequations}\label{pholoframe}
\begin{gather}
\bfn_1=\frac{h_{\Sigma}(\bft_1,\bft_1)}{|h|};\qquad
\bfn_2=\frac{h_{\Sigma}(\bft_1,\bft_2)}{|h|};\\
\bfb_1=\bft_1\times\frac{h_{\Sigma}(\bft_1,\bft_1)}{|h|};\qquad\bfb_2=\bft_2\times\frac{h_{\Sigma}(\bft_2,\bft_2)}{|h|};\\
\bft=\frac{1}{2}(\bft_1-i\bft_2);\qquad\bfn=\frac{1}{2}(\bfn_1-i\bfn_2);\qquad\bfb=\frac{1}{2}(\bfb_1-i\bfb_2).
\end{gather}
\end{subequations}
In this way, $\bft$ spans $T^{1,0}U$, $\bfn$ spans $N_1U$ and 
 $\bfb$ spans $N_2U$.

\medskip

From \cite[Proposition 2.3 \& $\S$4]{Bryant1},  we may
use a complex matrix presentation of $\g2$ and write the map $\gg:\GGG_2\rightarrow
\GL(7,\C)$ as $\gg=(\bfu\;\bff\;\bar{\bff})$, where $\bff=(\bff_1\;\bff_2\;\bff_3)\in\text{M}_{7\times 3}(\C)$, 
to derive 
 the structure equations of $\GGG_2$.  
 Over $\Sigma$, we recognise $\bfu$ as a point in $\Sigma$ and 
$\bff$ as a unitary frame for $T\mathcal{S}^6|_{\Sigma}$.  Thus, on the adapted frame bundle over $\Sigma$, we can set  $\theta_2=\theta_3=0$.  We deduce the following result. 

\begin{prop} For a $3\times 1$ vector of complex-valued 1-forms $\theta=(\theta_1,\theta_2,\theta_3)^{\rm T}$ and a 
$3\times 3$ skew-Hermitian matrix of 1-forms $\kappa=(\kappa_{ij})$ satisfying $\Tr \kappa=0$, 
the structure equations of $\GGG_2$ can be written as:
\begin{subequations}\label{pholoeqgen}
\begin{align}
\d\bfu&=-2i\bff\theta+2i\bar{\bff}\bar{\theta};\\
\d\bff&=-i\bfu\bar{\theta}^{\rm T}+\bff\kappa+\bar{\bff}[\theta];\\
\d\theta&=-\kappa\w\theta+[\bar{\theta}]\w\bar{\theta};\\
\d\kappa&=-\kappa\w\kappa+3\theta\w\bar{\theta}^{\rm T}-\theta^{\rm T}\w\bar{\theta}\Id_3,
\end{align}
\end{subequations}
where $\Id_3$ is the $3\times 3$ identity matrix.

On the adapted frame bundle of $\Sigma$, there exist holomorphic functions $k_2$ and $k_3$ such that $\kappa_{21}=k_2\theta_1$ 
and $\kappa_{31}=k_3\theta_1$, and the structure equations become:
\begin{subequations}\label{pholoeq}
\begin{align}
\d\bfu&=-2i\bff_1\theta_1+2i\bar{\bff}_1\bar{\theta}_1;\\
\d\bff_1&=-i\bfu\bar{\theta}_1-\bff_1(\kappa_{22}+\kappa_{33})+k_2\bff_2\theta_1+k_3\bff_3\theta_1;\\
\d\bff_2&=-\bar{k}_2\bff_1\bar{\theta}_1+\bff_2\kappa_{22}+\bff_3\kappa_{32}-\bar{\bff}_3\theta_1;\\
\d\bff_3&=-\bar{k}_3\bff_1\bar{\theta}_1-\bff_2\bar{\kappa}_{32}+\bff_3\kappa_{33}+\bar{\bff}_2\theta_1;\label{pholoeq4}\\
\d\theta_1&=(\kappa_{22}+\kappa_{33})\w\theta_1;\label{pholoeq5}\\
\d\kappa_{22}&=\left(|k_2|^2-1\right)\theta_1\w\bar{\theta}_1-\kappa_{32}\w\bar{\kappa}_{32};\\
\d\kappa_{33}&=\left(|k_3|^2-1\right)\theta_1\w\bar{\theta}_1+\kappa_{32}\w\bar{\kappa}_{32};\label{pholoeq8}\\
\d(k_2\theta_1)&=-\big(k_2(2\kappa_{22}+\kappa_{33})-k_3\bar{\kappa}_{32}\big)\w\theta_1;\\
\d(k_3\theta_1)&=-\big(k_3(\kappa_{22}+2\kappa_{33})+k_2\kappa_{32}\big)\w\theta_1;\label{pholoeq10}\\
\label{pholoeq11}\d\kappa_{32}&=\bar{k}_2k_3\theta_1\w\bar{\theta}_1+(\kappa_{22}-\kappa_{33})\w\kappa_{32}.
\end{align}
\end{subequations}
\end{prop}
\noindent
 
We can interpret the functions $(k_2,k_3)$ as the \emph{second fundamental form} of the 
pseudoholomorphic curve.  Indeed, by \cite[Lemma 4.4]{Bryant1}, $(k_2,k_3)=0$ if and only if $\Sigma$ lies in a totally geodesic 
$\mathcal{S}^2$, which is the intersection of a linear associative 3-plane in $\Im\O$ with $\mathcal{S}^6$.  We can observe this ourselves 
using \eq{pholoeq5}-\eq{pholoeq8}.  

Suppose that $\Sigma$ is non-totally geodesic. If we adapt frames further, as suggested earlier, so that 
$\bff_2$ and $\bff_3$ span $N_1\Sigma$ and $N_2\Sigma$, we find that $k_3=0$.  Moreover, by
\cite[Lemma 4.5]{Bryant1}, there exists a holomorphic function $k_1$ such that $\kappa_{32}=k_1\theta_1$.  
Following \cite{Bryant1}, we call $k_1$ the 
\emph{torsion} of $\Sigma$.  The pseudoholomorphic curves with 
\emph{null-torsion} (also called \emph{superminimal} in the language of integrable systems) exhibit a rich geometry: in fact, every Riemann surface can be embedded as a null-torsion curve in $\mathcal{S}^6$ 
by \cite[Theorem 4.10]{Bryant1}.  It also clear from \eq{pholoeq} that $\Sigma$ lies in a 
totally geodesic $\mathcal{S}^5$ if and only if the torsion is constant and satisfies $|k_1|=1$.  In this case, $\Sigma\subseteq\mathcal{S}^5$ 
is a \emph{minimal Legendrian} surface.

By \cite[Lemma 4.3]{Bolton}, pseudoholomorphic curves in $\mathcal{S}^6$ split into four types:
linearly full and null-torsion; linearly full with non-zero torsion; linearly full in a totally geodesic 
$\mathcal{S}^5$ (and necessarily with non-zero torsion); and totally geodesic.

\section[Ruled and quasi-ruled Lagrangian submanifolds]{Ruled and quasi-ruled Lagrangian\\ submanifolds}\label{s4}

We consider Lagrangians in $\mathcal{S}^6$ that are ruled by circles of constant radius.  
However, we reserve the notation `ruled' for the case where the circles are geodesics.

\begin{dfn}
Let $L$ be a 3-dimensional submanifold of $\mathcal{S}^6$ and let $\lambda\in(0,1]$ be constant.  
A \emph{$\lambda$-ruling} of $L$ is a pair $(\Sigma,\pi)$ where $\pi:L\rightarrow\Sigma$ is a smooth fibration of $L$  
over a 2-manifold $\Sigma$ by oriented circles of radius $\lambda$ in $\mathcal{S}^6$.  
We say that $(L,\Sigma,\pi)$ is \emph{ruled} if $(\Sigma,\pi)$ is a 1-ruling of $L$, and that it is 
\emph{quasi-ruled} if $(\Sigma,\pi)$ is a $\lambda$-ruling of $L$ for $\lambda\in(0,1)$.
\end{dfn}

We begin by describing the `second order' condition on a Lagrangian corresponding to the ruled or quasi-ruled condition.

\begin{lemma}\label{rcubic}
Let $L$ be a Lagrangian in $\mathcal{S}^6$ with a $\lambda$-ruling and let $C_L$ be its fundamental cubic.  
There exists an orthonormal frame $\{\bfe_1,\bfe_2,\bfe_3\}$ of $L$, with dual coframe $\{2\omega_1,2\omega_2,2\omega_3\}$, 
 such that $\bfe_1$ is the direction of the $\lambda$-ruling and 
\begin{align*}
&C_L=C(r,s,a,b)\\
&=r\omega_1(2\omega_1^2-3\omega_2^2-3\omega_3^2)+3s\omega_1(\omega_2^2-\omega_3^2)+
a\omega_2(\omega_2^2-3\omega_3^2)+b\omega_3(3\omega_2^2-\omega_3^2),
\end{align*}
where
 $r=\frac{4}{\lambda}\,\sqrt{1-\lambda^2}$ 
and $s,a,b$ are functions.  Moreover, if $s=0$ we can choose $b=0$.
\end{lemma}

\begin{proof}
Since $L$ has a $\lambda$-ruling, $h_L(\bfe_1,\bfe_1)=rJ\bfe_1$ for some constant $r$.  Therefore, it is easy to 
see that $C_L=C(r,s,a,b)$.  To determine the relationship between $r$ and $\lambda$, one need only look at the first structure 
equations \eq{firstL} for $\omega_2=\omega_3=0$:
\begin{equation*}\label{circle}
\d\bfx =2\bfe_1\omega_1;\quad \d\bfe_1=-2\bfx\omega_1+\frac{r}{2}\,J\bfe_1\omega_1;\quad \d J\bfe_1=\frac{r}{2}\,\bfe_1.
\end{equation*}
These are the equations for a circle of radius $4(16+r^2)^{-\frac{1}{2}}$, which must necessarily equal $\lambda$.  The formula 
for $r$ follows.
If $s=0$, then we may use the $\SO(2)$ subgroup of $\SO(3)$ that fixes $\bfe_1$ to set $b=0$.
\end{proof}

In $\S$\ref{s1} we mentioned the relationship between ruled Lagrangians in $\mathcal{S}^6$ and \emph{2-ruled} 
coassociative 4-folds in $\Im\O$.  We now briefly define the latter.

\begin{dfn}
Let $N$ be a 4-dimensional submanifold of $\Im\O$.  A \emph{2-ruling} of $N$ is a pair $(\Sigma,\pi)$ 
 where $\pi:N\rightarrow \Sigma$ is a smooth fibration of $N$ over a 2-manifold $\Sigma$ by oriented affine 2-planes in $\Im\O$.  
We say that $(N,\Sigma,\pi)$ is \emph{2-ruled} if $(\Sigma,\pi)$ is a 2-ruling of $N$.
\end{dfn}

We now make an elementary observation.

\begin{lemma}\label{r2rlemma}
A 4-dimensional coassociative cone in $\Im\O$ is 2-ruled if and only if its Lagrangian link in $\mathcal{S}^6$ is ruled.
\end{lemma}
 
By \cite[Proposition 7.2]{Fox2}, there is a correspondence between 2-ruled coassociative cones in $\Im\O$ and 
certain surfaces in the Grassmannian of oriented 2-planes in $\Im\O$, $\Gr_+(2,\Im\O)$.  Notice that $\Gr_+(2,\Im\O)$ is  
naturally isomorphic to the space $\mathcal{C}$ of oriented geodesic circles in $\mathcal{S}^6$:  simply identify
an oriented 2-plane in $\Im\O$ with its intersection with $\mathcal{S}^6$.  Therefore, a surface 
$\phi:\Sigma\rightarrow\mathcal{C}$ can be written as 
\begin{equation}\label{littlephi}
\phi(p)=\big(\bfv_1(p),\bfv_2(p)\big),
\end{equation}
where $\bfv_1,\bfv_2:\Sigma\rightarrow\mathcal{S}^6$ are everywhere orthogonal, and so define an oriented basis for a 2-plane at each point.
  We then associate a map $\Phi:\Sigma\times [0,2\pi)\rightarrow\mathcal{S}^6$ to $\phi$, whose image is a ruled 3-dimensional submanifold, in the obvious way:
\begin{equation}\label{bigphi}
\Phi(p,t)=\bfv_1(p)\cos t+\bfv_2(p)\sin t.
\end{equation}

To state our result we need to define \emph{almost CR-structures}.

\begin{dfn}
An \emph{almost CR-structure} on a manifold $M$ is a pair $(E,I)$ where $E$ is an even-dimensional subbundle of $TM$ and 
$I$ is a complex structure map on $E$.  An almost CR-structure $(E,I)$ is \emph{Levi-flat} if for every 1-form $\eta$ on $M$ 
such that $\eta|_E=0$, $\d\eta$ vanishes on all complex lines in $E$.  

Let $M$ be endowed with an almost CR-structure $(E,I)$.  A surface $\Sigma$ in $M$ is a \emph{CR-holomorphic curve} if 
$T_p\Sigma$ is a complex line in $E$ for all $p\in\Sigma$.
\end{dfn}

Using Lemma \ref{r2rlemma} and \cite[Proposition 7.2]{Fox2}, we deduce the following.

\begin{prop}  Let $\mathcal{C}$ be the space of oriented geodesic circles in $\mathcal{S}^6$.   
There is a complex structure $I$ on an 8-plane bundle $E\subseteq T\mathcal{C}$ such that:
\begin{itemize}
\item[\emph{(i)}] $(E,I)$ is a real analytic, Levi-flat, $\GGG_2$-invariant almost CR-structure on $\mathcal{C}$;
\item[\emph{(ii)}] every CR-holomorphic curve $\phi:\Sigma\rightarrow\mathcal{C}$ as in \eq{littlephi} 
defines a ruled Lagrangian in 
$\mathcal{S}^6$ via $\Phi:\Sigma\times [0,2\pi)\rightarrow\mathcal{S}^6$ as in \eq{bigphi}; and
\item[\emph{(iii)}] every ruled Lagrangian $(L,\Sigma,\pi)$ in $\mathcal{S}^6$ defines a CR-holomorphic curve in 
$\mathcal{C}$, $\phi:\Sigma\rightarrow\mathcal{C}$, where $\phi(p)=\pi^{-1}(p)$.
\end{itemize}
\end{prop}

\begin{remarks}
This is the natural analogue of the characterisation of ruled special Lagrangians in $\C^3$ given in 
\cite[Theorem 6]{Bryant2}.   We should stress that our proposition is little more 
than a repackaging of the material given in \cite[Proposition 7.2]{Fox2}.
\end{remarks}


It follows from \cite[$\S$8]{Fox2} that CR-holomorphic curves in $\mathcal{C}$ 
are lifts of pseudoholomorphic curves in $\mathcal{S}^6$. 
 We shall show that, in fact, every linearly full 
ruled or quasi-ruled Lagrangian is locally a \emph{tube} about a pseudoholomorphic curve defined using \emph{holomorphic} data.

\begin{dfn}\label{tubes}
Let $\bfu:\Sigma\rightarrow\mathcal{S}^6$ be an immersed surface.  
Let $\Pi$ be an oriented 2-plane subbundle of $\bfu^*(T\mathcal{S}^6)$ and let $\mathcal{U}(\Pi)=\{\bfv\in\Pi\,:\,|\bfv|=1\}$. 
Let $\gamma\in(0,\frac{\pi}{2}]$ be a constant. Define $\bfx_{\gamma}:\mathcal{U}(\Pi) 
\rightarrow\mathcal{S}^6$ by
$$\bfx_{\gamma}(\bfv)=\cos\gamma\bfu+\sin\gamma\bfv
.$$
When $\bfx_\gamma$ is an immersion, we say that its image is a \emph{tube of radius $\gamma$} (in $\Pi$) 
about $\Sigma$.  Clearly, a tube of radius $\gamma$ has a 
$\sin\gamma$-ruling and is thus ruled if $\gamma=\frac{\pi}{2}$ and quasi-ruled otherwise.
\end{dfn}

\begin{remark}
This is a generalisation of the tubes about pseudoholomorphic curves first studied in \cite{Ejiri2}.  
\end{remark}


\section[The fundamental cubic and the Gauss equation]{The fundamental cubic and the Gauss\\ equation}\label{s5}

In this section we discuss the possible pointwise symmetries of the fundamental cubic and then 
analyse the Riemann curvature tensors satisfying the Gauss equation.  Though this is strictly more than we 
require, we feel it is inherently interesting, and that it helps expose the links between the symmetries of 
the fundamental cubic and the curvature conditions studied by other authors.

 We remarked earlier that the fundamental cubic $C_L$ of a Lagrangian $L$ in $\mathcal{S}^6$ naturally defines, at each point, a homogeneous cubic $h=h_{ijk}x_ix_jx_k$ which is harmonic 
 in the variables $(x_1,x_2,x_3)=(x,y,z)$ on $\R^3$.  
 We now exploit this fact, since these cubics on $\R^3$ are classified according to their stabilizer in $\SO(3)$ in \cite[Proposition 1]{Bryant2}.

\begin{prop}\label{s2prop2a}
Let $\mathcal{H}^3(\R^3)$ denote the space of homogeneous harmonic cubics on $\R^3$.  The stabilizer of 
$h\in\mathcal{H}^3(\R^3)$ in $\SO(3)$ is nontrivial if and only if $h$ lies on the $\SO(3)$-orbit of exactly one of the 
polynomials in Table \ref{symtable} below.

\begin{table}[ht]
\begin{center}
\begin{tabular}{|c|c|c|}
\hline
Cubic in $\mathcal{H}^3(\R^3)$ & Parameter(s) & Stabilizer\\
\hline
$0$ & & $\SO(3)$\\
$rx(2x^2-3y^2-3z^2)$ & $r>0$ & $\SO(2)$\\
$3sx(y^2-z^2)$ & $s>0$ & $\AAA_4$\\
$ay(y^2-3z^2)$ & $a>0$ & $\SSS_3$\\
$rx(2x^2-3y^2-3z^2)+ay(y^2-3z^2)$ & $r,a>0,\,r\sqrt{2}\neq a$ & $\Z_3$\\
$rx(2x^2-3y^2-3z^2)+3sx(y^2-z^2)$ & $r,s>0,\,r\neq s$ & $\Z_2$\\
\hline
\end{tabular}
\caption{Homogeneous harmonic cubics on $\R^3$ with symmetries}\label{symtable}
\end{center}
\end{table}
\end{prop}


\begin{remarks}
The exclusions $r\sqrt{2}\neq a$ and $r\neq s$ in Table \ref{symtable} occur for the following reasons.
 The cubic  $rx(2x^2-3y^2-3z^2)+r\sqrt{2}y(y^2-3z^2)$ lies on the $\SO(3)$-orbit of $3r\sqrt{3}x(y^2-z^2)$, which has $\AAA_4$-stabilizer, 
and $rx(2x^2-3y^2-3z^2)+3rx(y^2-z^2)$ lies on the $\SO(3)$-orbit of $2ry(y^2-3z^2)$, which has $\SSS_3$-stabilizer.
\end{remarks}

We now study whether a given Riemann curvature tensor can arise as a quadratic in the coefficients of 
a homogeneous harmonic cubic, as specified by the Gauss equation \eq{Ga}.  Since the Riemann curvature tensor is a
 slightly unwieldy algebraic object we simplify matters using the following definition.

\begin{dfn}\label{Kdfn} Let $R_{abcd}$ be
the Riemann curvature tensor of the metric $g_L$ on $L$.  
For cyclic permutations $(i\; j\; k)$ of $(1\; 2\; 3)$ define, 
using the ``omitted index'' rule, $K_{ii}=R_{jkjk}-1$ and $K_{ij}=R_{jkki}$.  The resulting tensor $K_{ab}$ 
can be thought of as a $3\times 3$ symmetric matrix $K$.
Define a map from $\mathcal{H}^3(\R^3)$ to $\Sym_3(\R)$, $h\mapsto K(h)$, by
\begin{equation}\label{Gmap}
K(h)_{ii}=\sum_q h_{jjq}h_{kkq}-h_{jkq}^2\quad\text{and}\quad K(h)_{ij}=\sum_q h_{ikq}h_{kjq}-h_{ijq}h_{kkq};
\end{equation}
 i.e.~we use the Gauss equation \eq{Ga} for $h$.  We call this the \emph{Gauss map}.
\end{dfn} 

\begin{remark}
We can recover the Ricci tensor $R_{ab}$ 
from $K_{ab}$ by the formulae $R_{ii}=K_{jj}+K_{kk}+2$ and $R_{ij}=-K_{ij}$, again using cyclic permutations 
$(i\; j\; k)$ of $(1\; 2\; 3)$.
\end{remark}

This definition leads us to consider the set of $K(h)$ for $h\in\mathcal{H}^3(\R^3)$.  Notice that
$K(0)=0$ and $K(th)=t^2 K(h)$ for all $t\in\mathbb{R}$, 
so the image of the Gauss map is a (1-sided) cone in $\text{Sym}_3(\mathbb{R})$.  
We shall give a description of this cone using a rather ``brute force'' approach.  

As we have the freedom to transform the frame over $L$, we can apply $\SO(3)$ transformations to the source cubic or 
the target matrix in the Gauss map.   In particular, we can diagonalise $K(h)$ to 
$\diag(\lambda_1,\lambda_2,\lambda_3)$. We start by studying the case where $K(h)$ has distinct 
eigenvalues since the calculations here (though still ugly!) are more straightforward.

\begin{prop}\label{diagprop}
Let $K=\diag(\lambda_1,\lambda_2,\lambda_3)$ with $\lambda_1>\lambda_2>\lambda_3$.  Let 
$\sigma(K)=\frac{1}{6}\big((\text{\emph{Tr}}\,K)^2-\text{\emph{Tr}}\,K^2\big)$.  
Then $K=K(h)$ for some $h\in\mathcal{H}^3(\mathbb{R}^3)$ as in \eq{Gmap} if and only if $\Tr K< 0$, $\sigma(K)> 0$ and 
$\lambda_1^2\leq\sigma(K)$.  
\end{prop}

\begin{proof}
If we let $\|h\|^2=h_{ijk}h_{ijk}$, using summation notation, then it is easy to calculate that 
\begin{equation}\label{trkeq}
\Tr K(h)=-\frac{1}{2}\,\|h\|^2.  
\end{equation}
Therefore, $\Tr K(h)<0$ as $K(h)$ is necessarily non-zero.

The equation $K=K(h)$ has a solution if and only if a quadratic in the coefficients of $h$ 
has real solutions.  Calculation shows that the generic solutions are, for $r$ a real parameter:
\begin{align*}
h_{331}&=r;\displaybreak[0]\\
h_{221}&=\frac{\lambda_1-\lambda_3}{\lambda_1-\lambda_2}\,r;\displaybreak[0]\\
h_{112}&=\pm\left(\frac{(\lambda_1-\lambda_2)^2\big(\lambda_3^2-\sigma(K)\big)-
2(\lambda_1-\lambda_3)^3r^2}{2(\lambda_1-\lambda_2)^2(\lambda_2-\lambda_3)}\right)^{\frac{1}{2}};\displaybreak[0]\\
h_{332}&=-\frac{\lambda_1-\lambda_2}{\lambda_2-\lambda_3}\,h_{112};\displaybreak[0]\\
h_{113}&=\pm\left(\frac{\big(\sigma(K)-\lambda_2^2\big)+2(\lambda_1-\lambda_2)r^2}{2(\lambda_2-\lambda_3)}\right)^{\frac{1}{2}};\displaybreak[0]\\
h_{223}&=\frac{\lambda_1-\lambda_3}{\lambda_2-\lambda_3}\,h_{113};\;\,\text{and}\displaybreak[0]\\
h_{123}&=0.
\end{align*}

Notice that the parameter $r$ is constrained such that if we let 
$s=2(\lambda_1-\lambda_3)^3(\lambda_1-\lambda_2)r^2\geq0$, then 
$$(\lambda_1-\lambda_3)^3\big(\lambda_2^2-\sigma(K)\big)\leq s\leq (\lambda_1-\lambda_2)^3\big(\lambda_3^2-\sigma(K)\big).$$
The difference in the upper and lower bounds is $(\lambda_2-\lambda_3)^3\big(\sigma(K)-\lambda_1^2\big)$, so 
we must have that $\lambda_1^2\leq\sigma(K)$ for real solutions to exist.  

The condition $\lambda_1^2\leq\sigma(K)$ clearly forces $\sigma(K)\geq 0$.  If $\sigma(K)=0$, 
$\lambda_1=0$ and
\begin{equation}\label{sigma0}
3\sigma(K)=\lambda_1\lambda_2+\lambda_2\lambda_3+\lambda_3\lambda_1=\lambda_2\lambda_3=0,
\end{equation}
so at least two of the eigenvalues are zero, our required contradiction.   

Further, real solutions exist for $r\neq0$ only if 
$\lambda_3^2>\sigma(K)$, and for $r=0$ we need $\lambda_3^2\geq\sigma(K)$ and
$\lambda_2^2\leq\sigma(K)$.  However, we see that
\begin{align*}
3\big(\lambda_3^2-\sigma(K)\big)&=\lambda_3(\lambda_3-\lambda_2)+\lambda_3(\lambda_3-\lambda_1)+\lambda_3^2-\lambda_1
\lambda_2.
\end{align*}
Thus, if $\lambda_1\lambda_2\leq 0$ then $\lambda_3^2>\sigma(K)$.  If $\lambda_1\lambda_2>0$ then the fact that 
$\Tr K<0$ forces $\lambda_3^2>\lambda_1\lambda_2$, and hence $\lambda_3^2>\sigma(K)$ as well.  Thus 
$\lambda_3^2>\sigma(K)$ always holds.

The only question left is whether $\lambda_2^2\leq\sigma(K)$ is an additional constraint.  
Since $\lambda_3=\Tr K-\lambda_1-\lambda_2$ and $\lambda_3^2<(\Tr K)^2$, because $\sigma(K)>0$, we 
must have that $\lambda_1+\lambda_2<0$.  Hence $\lambda_1^2<\lambda_2^2$, so $\lambda_2^2\leq\sigma(K)$ implies 
$\lambda_1^2\leq\sigma(K)$.  
\end{proof}

Before stating our next result, we notice from Table \ref{symtable} 
that there are two families of cubics $h$ with $\Z_2$-stabilizer, given by $r>s$ and $r<s$.  This leads to two 
 distinct families of corresponding matrices $K(h)$.  

\begin{prop}\label{Kimage}
Let $K\in\Sym_3(\R)$, let $\sigma(K)=\frac{1}{6}\big((\Tr K)^2-\Tr K^2\big)$ and let $\lambda_1\geq\lambda_2\geq\lambda_3$ be the eigenvalues of $K$.
Then $K=K(h)$ for some $h\in\mathcal{H}^3(\R^3)$ as in \eq{Gmap} if and only if $\Tr K\leq 0$, $\sigma(K)\geq 0$ 
and $\lambda_1^2\leq\sigma(K)$.

Moreover, the following hold for $K=K(h)$, $h\in\mathcal{H}^3(\R^3)$.
\begin{itemize}
\item[\emph{(i)}] $\Tr K=0$ if and only if $K=h=0$. 
\item[\emph{(ii)}] $\sigma(K)=0$ and $K\neq 0$ if and only if $\lambda_1=\lambda_2=0>\lambda_3$, which is if and only if
$h$ has an $\SSS_3$-stabilizer.  
\item[\emph{(iii)}] $\lambda_1^2=\sigma(K)>0$ and $K$ has distinct eigenvalues only if 
$h$ has a $\Z_2$-stabilizer.  
\item[\emph{(iv)}] $\lambda_1^2=\sigma(K)>0$ and $K$ has exactly two repeated eigenvalues if and only if
$-\frac{1}{5}\,\Tr K=\lambda_1>\lambda_2=\lambda_3=\frac{3}{5}\,\Tr K$, which is if and only if $h$ has an $\SO(2)$-stabilizer.  
\item[\emph{(v)}] $\lambda_1^2=\sigma(K)>0$ and $K$ has three repeated eigenvalues if and only if 
$K=-\lambda \Id_3$ for some $\lambda>0$, which occurs if and only if $h$ has an $\AAA_4$-stabilizer.  
\end{itemize}
\end{prop}

\begin{note}
This result classifies all Riemann curvature tensors associated with the metrics of Lagrangians in $\mathcal{S}^6$. 
Furthermore, it also performs this classification for the case of special Lagrangian 3-folds in $\C^3$.  
This should help lead to a solution of the isometric embedding problem for these submanifolds.  
\end{note}

\begin{proof}
Since the set of $A\in\Sym_3(\R)$ with distinct eigenvalues is dense in $\Sym_3(\R)$, and every 
element of $\Sym_3(\R)$ is $\SO(3)$-equivalent to a diagonal matrix, we need only 
turn the strict inequalities in Proposition \ref{diagprop} to equalities to give our conditions.  

Part (i) follows from \eq{trkeq}.  For (ii), notice that $\sigma(K)=0$ forces $\lambda_1=0$ and \eq{sigma0} leads to 
$\lambda_2\lambda_3=0$.  The additional assumption that $K\neq 0$ means $\lambda_2=0$.  
Diagonalise $K$ to $\diag(\lambda_3,0,0)$.   We can explicitly calculate
the cubics $h$ which map to this diagonal matrix as:
$$ay(y^2-3z^2)+bz(3y^2-z^2)$$
where $a,b$ satisfy $2a^2+2b^2+\Tr K=0$.  Using $\SO(3)$ to set $b=0$, we see from Table \ref{symtable} that 
$h$ has an $\SSS_3$-stabilizer.  
  
For (iii), by diagonalising
$K=K(h)$ to $\diag(\lambda_1,\lambda_2,\lambda_3)$, we notice from the proof of Proposition \ref{diagprop} that 
we must have $h_{112}=h_{332}=h_{113}=h_{223}=h_{123}=0$.  Therefore 
this $h$ is, up to sign,
$$\left(\frac{p+q}{2}\right)x(2x^2-3y^2-3z^2)+3\left(\frac{q-p}{2}\right)x(y^2-z^2)$$
where $p=\sqrt{\frac{1}{2}(\lambda_3-\Tr K)}$ and $q=\sqrt{\frac{1}{2}(\lambda_2-\Tr K)}$.  
Notice that $p$ and $q$ are both non-zero, since otherwise $3\sigma(K)=-\lambda_1^2<0$, a contradication.  Moreover,
$q^2-p^2=\frac{1}{2}(\lambda_2-\lambda_3)>0$. Thus, $h$ is a cubic with $\Z_2$-symmetry by Table \ref{symtable} 
with parameters $r=\frac{1}{2}(p+q)$ and $s=\frac{1}{2}(q-p)$ satisfying $r>s>0$.  

A cubic $h$ with $\Z_2$-symmetry, with parameters $r,s$ as in 
Table \ref{symtable}, defines a matrix $K(h)$ with eigenvalues $r^2-s^2$ and $-3r^2-s^2\pm 4rs$.  Thus, $K(h)$ has $\sigma(K)=(r^2-s^2)^2$.  
If the eigenvalues of $K(h)$ are $\lambda_1>\lambda_2>\lambda_3$, then $\lambda_1=r^2-s^2$ if $r>s$ and 
$\lambda_2=r^2-s^2$ if $r<s$.  This proves (iii). 

Now suppose there are at least two repeated eigenvalues for $K$ with $\lambda_1^2=\sigma(K)>0$.   
If $\lambda_1>\lambda_2=\lambda_3$, 
$\lambda_1+2\lambda_2=\Tr K$ and $3\lambda_1^2=\lambda_2(2\lambda_1+\lambda_2)>0$.
We quickly see that $\lambda_1=-\frac{1}{5}\,\Tr K$ and $\lambda_2=\lambda_3=\frac{3}{5}\,\Tr K$.  Again
by diagonalising $K$ we can solve for $h$ as:
$$rx(2x^2-3y^2-3z^2)$$
where $5r^2+\Tr K=0$. This $h$ has $\SO(2)$-symmetry by Table \ref{symtable}.  
If $\lambda_1=\lambda_2>\lambda_3$, the formulae $2\lambda_1+\lambda_3=
\Tr K$, $3\lambda_1^2=\lambda_1(2\lambda_3+\lambda_1)$ and $\Tr K<0$ imply that $\lambda_1=\lambda_2=0$,
 but this has $\sigma(K)=0$, a contradiction.  Part (iv) follows.  

Finally, suppose $K=-3p^2I$ for $p> 0$, which clearly has $\lambda_1^2=\sigma(K)>0$.  
Clearly, $h=3\sqrt{3}p x(y^2-z^2)$ maps to $K$ and has an $\AAA_4$-stabilizer.  However, 
we can explicitly calculate that the fibre of \eq{Gmap} at 
$K$ contains the aforementioned cubic together with cubics of the form:
$$ px(2x^2-3y^2-3z^2)+\sqrt{2 p^2-q^2}\,y(y^2-3z^2)+q z(3y^2-z^2)$$
for $q$ satisfying $q^2\leq 2p^2$.  Since we can use an $\SO(3)$ transformation to set $q=0$, the remarks 
before Table \ref{symtable} show that these cubics also have $\AAA_4$-stabilizer. 
\end{proof}

\begin{remark}
By the proof of Proposition \ref{Kimage}(iii), $K(h)$ has eigenvalues $\lambda_1>\lambda_2>\lambda_3$
satisfying either $\lambda_1^2=\sigma(K)$ or $\lambda_2^2=\sigma(K)$ if and only if $h$ has a $\Z_2$-stabilizer.
\end{remark}

\begin{prop}\label{repeat}
Let $K\in\text{\emph{Sym}}_3(\mathbb{R})$ and use the notation of Proposition \ref{Kimage}.  Suppose that
$\lambda_1^2<\sigma(K)$, $\Tr K<0$ and that $K=K(h)$ as in \eq{Gmap}.  
Then $K$ has exactly two repeated eigenvalues if and only if $h$ has a $\Z_3$-stabilizer in $\SO(3)$.  
\end{prop}

\begin{proof}  Suppose first that the eigenvalues satisfy $\lambda_1>\lambda_2=\lambda_3$.
By assumption $3\lambda_1^2< \lambda_2(2\lambda_1+\lambda_2)$.  
Using $\lambda_1+2\lambda_2=\Tr K$, we see that $\frac{3}{5}\,\Tr K<\lambda_2<\frac{1}{3}\,\Tr K$.
We can write $\lambda_2=\lambda_3=-3r^2$ for some $r>0$, since
$\lambda_2<0$.  We know that $\Tr K< -5r^2$, so there exists $a>0$ such that $\Tr K=-5r^2-2a^2$.  
One quickly sees that $\lambda_1=r^2-2a^2$ and that $\lambda_2<\frac{1}{3}\,\Tr K$ if and only if $a<r\sqrt{2}$.
Diagonalizing $K$ we recognise it as the image of cubics of the form 
\begin{equation}\label{z3cubic}
rx(2x^2-3y^2-3z^2)+\sqrt{a^2-b^2}\,y(y^2-3z^2)+bz(3y^2-z^2),
\end{equation}
where $b^2\leq a^2<2r^2$.  Using $\SO(3)$ to set $b=0$, we see from Table \ref{symtable} that these cubics 
have $\mathbb{Z}_3$-stabilizer.
 If $\lambda_1=\lambda_2>\lambda_3$, similar arguments show that $K$ is the image of 
cubics of the form \eq{z3cubic}, but now with $a>r\sqrt{2}$. 
\end{proof}

As a neat corollary, by analysing the proofs of Propositions \ref{diagprop}-\ref{repeat}
, we can collect together certain of our results concerning the Gauss map in
 terms of stabilizers of the fundamental cubic in $\SO(3)$.

\begin{cor}\label{htable}
Let $K\in\Sym_3(\R)$ with eigenvalues $\lambda_1\geq\lambda_2\geq\lambda_3$ 
be such that $K=K(h)$ as in \eq{Gmap}.  Use the notation of Proposition \ref{Kimage}.  
We can present Table \ref{fibretable} for stabilizers of $h$ in $\SO(3)$,
 properties of $K$ and the dimension of the fibres of the Gauss map \eq{Gmap} at $K$.  
For the case of zero stabilizer, the fibre of \eq{Gmap} has dimension 1 at generic $K$.

\begin{table}[h]
\begin{center}
\begin{tabular}{|c|c|c|}
\hline
Stabilizer of $h$ & Property of $K$ & Dimension\\
\hline
$\SO(3)$ & $K=0$ & 0 \\
$\SO(2)$ & $\lambda_1^2=\sigma(K)>0$, so $\lambda_1>\lambda_2=\lambda_3$ & 2\\
$\AAA_4$ & $K=-\lambda\Id_3$, $\lambda> 0$ & 3\\
$\SSS_3$ & $\lambda_1^2=\sigma(K)=0$, so $\lambda_1=\lambda_2=0>\lambda_3$ & 2\\
$\Z_3$ & $\lambda_1^2<\sigma(K)$, two eigenvalues are repeated & 1\\
$\Z_2$ & $\lambda_1>\lambda_2>\lambda_3$, $\lambda_1^2=\sigma(K)$ or $\lambda_2^2=\sigma(K)$ & 2
\\
$0$ & $\lambda_1>\lambda_2>\lambda_3$, $\lambda_1^2<\sigma(K)$ and $\lambda_2^2\neq\sigma(K)$ & 1\\
\hline
\end{tabular}
\caption{Fibres of the Gauss map and symmetries of the fundamental cubic}\label{fibretable}
\end{center}
\end{table}
\end{cor}



\section[Lagrangian submanifolds and fundamental cubics with symmetries]{Lagrangian submanifolds and fundamental\\ cubics with symmetries}\label{s6}

In this section we classify the families of Lagrangians in $\mathcal{S}^6$ 
whose fundamental cubic admits a pointwise symmetry.  This is mainly a detailed survey of results 
by other authors, though we also include new results and observations. 

\medskip

For this section let $L$ be a Lagrangian submanifold of $\mathcal{S}^6$ with fundamental cubic $C_L$ and suppose, for simplicity, that
it is connected.  We use the notation of $\S$\ref{s3subs1}.  In particular, recall that $\{2\omega_1,2\omega_2,2\omega_3\}$ defines 
an orthonormal coframe for $L$ and that $\alpha+[\omega]$ is the connection 1-form of the Levi--Civita connection of the
metric $g_L$ on $L$.  Since $\alpha$ is skew-symmetric, for convenience we shall write 
$$\alpha_1=\alpha_{23},\qquad \alpha_2=\alpha_{31},\qquad \alpha_3=\alpha_{12}.$$
We organise our results and examples by the possible pointwise stabilizers of $C_L$ as given in Table \ref{symtable}. 
 To rule out trivial cases we make the following definition.

\begin{dfn}
A Lagrangian $L$ in $\mathcal{S}^6$ is \emph{simple} if it is a totally geodesic $\mathcal{S}^3$.  
\end{dfn}

\subsection[SO(3)]{{\boldmath $\SO(3)$}}\label{s6subs1}

The stabilizer of $C_L$ in $\SO(3)$ is all of $\SO(3)$ at every point if and only if $C_L=0$ by Proposition \ref{s2prop2a}.  
Using \eq{Ga} we see that 
$$\d(\alpha+[\omega])+(\alpha+[\omega])\w(\alpha+[\omega])=4\omega\w\omega^{\rm T},$$
so $L$ has constant curvature $1$.  Thus, $L$ is totally geodesic and hence simple.

\begin{prop}\label{s4prop1}  A connected Lagrangian submanifold
of $\mathcal{S}^6$ whose fundamental cubic has an $\SO(3)$-stabilizer at each point is simple.
\end{prop}

\begin{ex}\textbf{(Simple case)}\label{simple}
A simple Lagrangian 
is the intersection of a linear coassociative 4-plane in $\Im\O$ with $\mathcal{S}^6$ by Proposition \ref{coasscones}.  
Since $\GGG_2$ acts transitively 
on the set of coassociative 4-planes with isotropy $\SO(4)$, a simple Lagrangian has $\SO(4)$-symmetry and is 
(up to $\GGG_2$ transformation)
$$L_0=\{y_1\eps_1+y_3\eps_3+y_5\eps_5+y_7\eps_7\,:\,y_1^2+y_3^2+y_5^2+y_7^2=1\},$$
recalling the basis $\eps_i$ for $\Im\O$.  
Notice that $L_0$ is trivially ruled.  Furthermore, $L_0$ is a tube radius 
$\frac{\pi}{2}$ about a totally geodesic 2-sphere, where the tube is defined using the standard Hopf fibration 
$\mathcal{S}^3\rightarrow\mathcal{S}^2$. 
\end{ex}

\subsection[SO(2)]{{\boldmath $\SO(2)$}}

To give a feel for later, more complicated, calculations which will often be suppressed, we 
go through this case in some detail.

If $C_L\neq 0$ has an $\SO(2)$-stabilizer at each point, 
then, by Proposition \ref{s2prop2a}, there exist an open dense subset $L^*$ of $L$ and 
some function $r:L^*\rightarrow\mathbb{R}^+$ 
such that $C_L=r\omega_1(2\omega_1^2-3\omega_2^2-3\omega_3^2)$ defines an $\SO(2)$-subbundle $\mathcal{F}$ of the 
adapted frame bundle over $L^*$.  
Since $\mathcal{F}$ is an $\SO(2)$-bundle there 
exist functions $t_{ij}$ for $i=2,3$, $j=1,2,3$ such that
$\alpha_2=t_{2j}\omega_j$ and $\alpha_3=t_{3j}\omega_j$, using summation notation. 
Moreover, there exist functions $r_i$ for $i=1,2,3$ such that $\d r=r_i\omega_i$.  

Define $\beta_{ij}$ in terms of $r$ and $\omega_i$ using \eq{secondLb}.  The equations \eq{secondLa} and \eq{secondLd}  then give  
$t_{21}=t_{31}=r_2=r_3=0$, $t_{22}=t_{33}=-\frac{1}{2}$, $t_{23}=-t_{32}=t$ and $r_1=-4rt$.  
Putting this information in \eq{secondLc} forces $t=0$ and $r=2\sqrt{5}$.  Thus, we may take 
$L^*=L$, and see that $C_L=2\sqrt{5}\omega_1(2\omega_1^2-3\omega_2^2-3\omega_3^2)$ over $L$.

The second structure equations we have so far are:
\begin{subequations}\label{secondso2}
\begin{gather}
\d\omega_1=\omega_2\wedge\omega_3;\!\qquad\!\d\omega_2=\omega_3\wedge(\alpha_1+\textstyle\frac{3}{2}\omega_1);\!\qquad\!
\d\omega_3=(\alpha_1+\textstyle\frac{3}{2}\omega_1)\wedge\omega_2;\\
\d(\alpha_1+\textstyle\frac{3}{2}\omega_1)=6\omega_2\wedge\omega_3.
\end{gather}
\end{subequations}
We see that the structure equations for $\omega_2$ and $\omega_3$ define a constant curvature 2-sphere. 
Equations \eq{firstL} with $\omega_2=\omega_3=0$ give: 
\begin{equation}\label{firstso2}
\d\mathbf{x}=2\mathbf{e}_1\omega_1;\qquad
\d\mathbf{e}_1=-2\mathbf{x}\omega_1+\sqrt{5}J\mathbf{e}_1\omega_1;\qquad
\d J\mathbf{e}_1=-\sqrt{5}\mathbf{e}_1\omega_1.
\end{equation}
Clearly, \eq{firstso2} defines a circle with radius $\frac{2}{3}$.

We also notice that the Lie derivative $\mathcal{L}_{\mathbf{e}_1}(\omega_1^2+\omega_2^2+\omega_3^2)=0$, so that
$\mathbf{e}_1$ is a Killing vector for the metric.  Thus, $L$ is homogeneous and topologically $\mathcal{S}^3$ 
fibred by circles over $\mathcal{S}^2$.  Moreover, by inspection of \eq{secondso2}, 
$L$ is an $\SU(2)$-orbit in $\mathcal{S}^6$ for some $\SU(2)$ subgroup $\GGG$ of $\GGG_2$.  
 Clearly, $\GGG$ must have a commuting $\U(1)$ subgroup in $\GGG_2$ because of the circle
 fibration.  Calculating the eigenvalues of the generators of the Lie algebra $\mathfrak{g}$ of $\GGG$, 
 we recognise its action on $\Im\O\cong\R^3\oplus\C^2$ as $\SU(2)$ on $\C^2$ and $\SO(3)$ on 
$\R^3$.  Here, $\R^3$ is spanned by $\{\eps_1,\eps_2,\eps_3\}$ and $\C^2$ by $\{\eps_4+i\eps_6,\eps_5+i\eps_7\}$.  Explicitly, the generators of $\mathfrak{g}$ are:
\begin{align*}
U_1&=-2E_{23}+E_{45}+E_{67};\\
U_2&=-2E_{31}+E_{46}-E_{57};\\
U_3&=-2E_{12}-E_{47}-E_{56},
\end{align*}
where 
\begin{equation}
\begin{gathered}\label{elements}
E_{ij}(\eps_k)=\left\{\begin{array}{lll} \eps_i & & \text{if $j=k$,}\\ -\eps_j & & \text{if $k=i$,}\\ 
0 & & \text{otherwise.}
\end{array}\right.
\end{gathered}
\end{equation}
  Harvey and Lawson \cite[Theorem IV.3.2]{HarLaw} classify the coassociative submanifolds invariant under this
$\SU(2)$ action, hence the Lagrangian submanifolds which are orbits of this action on $\mathcal{S}^6$.  This
result also follows from \cite[Theorem 4.1]{Mashimo}. 

\begin{ex}\textbf{(A ``squashed'' 3-sphere)}\label{s4ex2}  Let $\{\eps_1,\ldots,\eps_7\}$ be the basis of $\Im\O$ given at 
the start of $\S$\ref{s3}.  
The 3-dimensional submanifold of $\mathcal{S}^6$ given by
$$
L_1=\left\{\frac{\sqrt{5}}{3}\,\bar{q}\eps_1 q+\frac{2}{3}\,q\eps_5\,:\,\text{$q\in\langle 1,\eps_1,
\eps_2,\eps_3\rangle_{\R}$ with $|q|=1$}\right\}$$
is Lagrangian.  Let $\mathcal{S}^3$ be the unit 3-sphere in $\R^4$ with coordinates $x_1,x_2,x_3,x_4$.  
Following \cite[Example 5.1]{Dillen}, define vector fields on $\mathcal{S}^3$ by:
\begin{align*}
\bfw_1&=x_2\bfv_1-x_1\bfv_2+x_4\bfv_3-x_3\bfv_4;\\
\bfw_2&=x_3\bfv_1-x_4\bfv_2-x_1\bfv_3+x_2\bfv_4;\\
\bfw_3&=x_4\bfv_1+x_3\bfv_2-x_2\bfv_3-x_1\bfv_4,
\end{align*}
where $\bfv_i=\frac{\partial}{\partial x_i}$.  These vectors form a basis for $T\mathcal{S}^3$, 
so we can define a metric $g_1$ on $\mathcal{S}^3$ by requiring that the $\bfw_i$ are orthogonal with respect
to $g_1$,
$$g_1(\bfw_1,\bfw_1)=\frac{4}{9}\quad\text{and}\quad g_1(\bfw_2,\bfw_2)=g_1(\bfw_3,\bfw_3)=\frac{8}{3}\,.$$
By \cite[Theorem 5.1]{Dillen}, $L_1$ is 
the isometric embedding of $(\mathcal{S}^3,g_1)$ via the map 
 $q\mapsto \frac{\sqrt{5}}{3}\bar{q}\eps_1q+\frac{2}{3}q\eps_5$,
 where we identify $\mathcal{S}^3$ with the unit sphere in $\langle 1,\eps_1,\eps_2,\eps_3\rangle_{\R}$.  
\end{ex}

\begin{remarks}
By scaling $L_1$, we recognise it as the graph of the Hopf map 
$\mathcal{S}^3\rightarrow\mathcal{S}^2$ given by $q\mapsto \frac{\sqrt{5}}{2} \bar{q}\eps_1 q$. Furthermore, although $L_1$ is 
an $\SU(2)$-orbit, it is in fact invariant under an action of $\U(2)$, since there is an extra commuting $\U(1)$-action 
which results from the circle fibration.
\end{remarks}

\begin{prop}\label{s4prop2}
The unique (up to rigid motion) connected, non-simple Lagrangian submanifold of $\mathcal{S}^6$ whose fundamental cubic 
has an $\SO(2)$-stabilizer at each point is $L_1$ given in Example \ref{s4ex2}.
\end{prop}

From Definition \ref{tubes} and Example \ref{s4ex2}, we see that $L_1$ is a tube of radius $\sin^{-1}(\frac{2}{3})$ 
about a totally geodesic 2-sphere.

\subsection[A4]{{\boldmath $\AAA_4$}}

If $C_L\neq 0$ has an $\AAA_4$-stabilizer at each point then, by Proposition \ref{s2prop2a}, there exist an open dense 
subset $L^*$ of $L$ and a function $s:L^*\rightarrow\mathbb{R}^+$, such that
$C_L=3s\omega_1(\omega_2^2-\omega_3^2)$ defines an $\AAA_4$-subbundle $\mathcal{F}$ of the adapted frame bundle over $L^*$.  
Therefore, there exist functions $t_{ij}$ on $\mathcal{F}$ such that $\alpha_i=t_{ij}\omega_j$.  
 Using \eq{secondL}, we find that $\alpha=-\frac{1}{2}[\omega]$ and $s=2\sqrt{15}$.  Therefore, we can take $L^*=L$, and
 $C_L=12\sqrt{15}\,\omega_1(\omega_2^2-\omega_3^2)$ over $L$.   Equation \eq{Ga} gives 
$$\d(\alpha+[\omega])+(\alpha+[\omega])\w(\alpha+[\omega])=\frac{1}{4}\,\omega\w\omega^{\rm T},$$
so $L$ has constant curvature $\frac{1}{16}$.  By Proposition \ref{Kimage}(v), our calculations here and 
in $\S$\ref{s6subs1} prove \cite[Theorem 2]{Ejiri}: the only constant curvature Lagrangian submanifolds of $\mathcal{S}^6$ have
 curvature 1 or $\frac{1}{16}$.

Further, by \cite[Lemma 2.5 \& Theorem 4.3(i)]{Mashimo}, $L$ must be, up to rigid motion, the orbit through $\eps_2$ of 
the 3-dimensional closed Lie subgroup $\GGG$ of $\GGG_2$ whose Lie algebra has the following generators:
\begin{subequations}\label{irractioneqns}
\begin{align}
U_1&=4E_{32}+2E_{54}+6E_{76};\\
U_2&=\sqrt{6}(2E_{51}-E_{62}+E_{73})+\sqrt{10}(E_{42}+E_{53});\\
U_3&=\sqrt{6}(2E_{41}+E_{63}+E_{72})+\sqrt{10}(E_{43}-E_{52}),
\end{align}
\end{subequations}
where $E_{ij}$ is given by \eq{elements}. 
 We have used the fact that all constant 
curvature $\frac{1}{16}$ Lagrangians are congruent up to $\GGG_2$ transformation.  The Lie group $\GGG$ is the $\SO(3)$
subgroup of $\SO(7)$ which acts irreducibly on $\R^7\cong\Im\O$.  We can interpret the group action as follows.

\begin{ex}\textbf{({\boldmath $\SO(3)$}-orbits 1)}\label{irractiondfn}\label{s4ex1}
Identify $\Im\O$ with the homogeneous harmonic cubics $\mathcal{H}^3(\R^3)$ on $\R^3$ by:
\begin{align*}
\eps_1&\mapsto \frac{\sqrt{10}}{10}\,x(2x^2-3y^2-3z^2);&&\displaybreak[0]\\
\eps_2&\mapsto -\sqrt{6}xyz;&
\eps_3&\mapsto \frac{\sqrt{6}}{2}\,x(y^2-z^2);\displaybreak[0]\\
\eps_4&\mapsto -\frac{\sqrt{15}}{10}\, y(4x^2-y^2-z^2);&
\eps_5&\mapsto -\frac{\sqrt{15}}{10}\,z(4x^2-y^2-z^2);\displaybreak[0]\\
\eps_6&\mapsto \frac{1}{2}\,y(y^2-3z^2);&
\eps_7&\mapsto -\frac{1}{2}\, z(z^2-3y^2).\displaybreak[0]
\end{align*}
Notice that the cubics above are of unit norm.  We then recognise the 
generators \eq{irractioneqns} of the Lie algebra of $\GGG$ as acting as
$$U_1=2\left(y\frac{\partial}{\partial z}-z\frac{\partial}{\partial y}\right), 
\;\, U_2=2\left(z\frac{\partial}{\partial x}-x\frac{\partial}{\partial z}\right), 
\;\, U_3=2\left(x\frac{\partial}{\partial y}-y\frac{\partial}{\partial x}\right).$$
Thus the action of $\GGG$ is simply the standard $\SO(3)$ action on $\mathcal{H}^3(\R^3)$, under this 
particular identification with $\Im\O$.

With this identification, the $\SO(3)$-orbit $L_2$ of $-\sqrt{6}xyz$ is Lagrangian in $\mathcal{S}^6$ and has constant curvature $\frac{1}{16}$.  Moreover, $L_2$ is diffeomorphic to $\SO(3)/\AAA_4$ by 
Table \ref{symtable}.  There is an explicit description of $L_2$ in \cite[Example 2]{Dillen} 
as a 24-fold isometric immersion of $\mathcal{S}^3(\frac{1}{16})$ in terms of harmonic polynomials of degree 6.
\end{ex}

\begin{remark} 
The main result of \cite{Dillen} is that $L_1$ and $L_2$ given in Examples \ref{s4ex2} and \ref{s4ex1} 
respectively, together with the simple example $L_0$, classify all Lagrangians whose sectional curvatures 
are bounded below by $\frac{1}{16}$.
\end{remark}

From our discussion, we can deduce the following result.

\begin{prop}\label{s4prop3} 
The unique (up to rigid motion) 
connected, non-simple Lagrangian submanifold of $\mathcal{S}^6$ 
whose fundamental cubic has
an $\AAA_4$-stabilizer at each point is $L_2$ as given in Example \ref{s4ex1}.
\end{prop}

\begin{remark}  In contrast, the only special Lagrangian 3-folds in $\C^3$ whose fundamental cubic has 
a pointwise $\AAA_4$-symmetry are 3-planes \cite[Theorem 2]{Bryant2}. 
\end{remark}

To see the ruling of $L_2$, we first notice, by the remarks after Proposition \ref{s2prop2a}, that $C_{L_2}$ is $\SO(3)$-equivalent to 
$2\sqrt{5}\omega_1(2\omega_1^2-3\omega_2^2-3\omega_3^2)+2\sqrt{10}\omega_2(\omega_2^2-3\omega_3^2)$.  The second
 structure equations are:
\begin{align*}
\d\omega_1=\omega_2\w\omega_3;\qquad\d\omega_2=\omega_3\w\omega_1;\qquad\d\omega_3=\omega_1\w\omega_2.
\end{align*}
By inspection, the equations for $\omega_2$ and $\omega_3$ define a 2-sphere of constant curvature.  
Moreover, the first structure equations with $\omega_2=\omega_3=0$ yield equations \eq{firstso2} as in the $\SO(2)$-stabilizer
case and hence define a circle of radius $\frac{2}{3}$.  Thus $L_2$ has a $\frac{2}{3}$-ruling over a constant curvature $\mathcal{S}^2$. 
 
This 2-sphere cannot be totally geodesic, otherwise the corresponding Lagrangian would be $L_1$ given in Example \ref{s4ex2}.  
By \cite[Theorem B]{Sekigawa}, the only possible constant Gauss curvatures of pseudoholomorphic curves in $\mathcal{S}^6$ are $0$, $\frac{1}{6}$ and $1$.  Therefore the $\mathcal{S}^2$ must have constant curvature $\frac{1}{6}$ and,  
by \cite[Theorem 4.6]{Bryant1}, have null-torsion.  Thus, this 2-sphere is congruent up to $\GGG_2$ transformations to 
the Bor\r{u}vka sphere $\mathcal{S}^2(\frac{1}{6})$ in $\mathcal{S}^6$, which is the orbit of $\epsilon_1$, or equivalently 
$\frac{\sqrt{10}}{10}\,x(2x^2-3y^2-3z^2)$, under the 
$\SO(3)$ action described in Example \ref{s4ex1}.

As noted in \cite[p.~123]{Ejiri2}, and as one could quickly verify using the structure equations, 
 $L_2$ is a tube of radius $\sin^{-1}(\frac{2}{3})$ about $\mathcal{S}^2(\frac{1}{6})$ in the second 
normal bundle.  We shall see that constructing a quasi-ruled Lagrangian tube about a non-totally geodesic pseudoholomorphic curve
in the second normal bundle is possible if and only if the curve has null-torsion and the radius is $\sin^{-1}(\frac{2}{3})$.

\subsection[S3]{{\boldmath $\SSS_3$}}

Suppose $C_L\neq 0$ has a pointwise $\SSS_3$-stabilizer. Then there is an open dense subset $L^*$ of $L$ and a function 
$a:L^*\rightarrow\R^+$ such that
$C_L=a\omega_2(\omega_2^2-3\omega_3^2)$ defines an $\SSS_3$-subbundle $\mathcal{F}$
 of the adapted frame bundle over $L^*$.  

By Proposition \ref{Kimage}(ii), the symmetric matrix $K$ associated with the Riemann curvature tensor of $L$, as defined in Definition \ref{Kdfn}, has a repeated eigenvalue of $0$.  Thus, by the remark following Definition \ref{Kdfn}, 
$L$ is \emph{quasi-Einstein}; that is, its Ricci tensor has repeated eigenvalues.  
Quasi-Einstein Lagrangians in $\mathcal{S}^6$ are classified in \cite{Deszcz} -- more on this later.  
We show that the Lagrangians whose fundamental cubic has a pointwise
$\SSS_3$-stabilizer at every point are in correspondence with the non-simple Lagrangians 
satisfying \emph{Chen's equality}.

\begin{dfn}\label{Chen}
Chen \cite{Chen} introduced a new Riemannian invariant $\delta_M$ to study submanifolds $M^n$ of spaces of constant curvature. 
 Explicitly, if $s$ is the sectional curvature
of $M$, $p\in M$, $\{\bfv_1,\ldots,\bfv_n\}$ is an orthonormal basis for $T_pM$ and $\mathcal{P}$ is the set of 2-planes in $T_pM$,
$$\delta_M(p)=\sum_{i<j}s(\bfv_i\w\bfv_j)-\inf_{\Pi\in \mathcal{P}}\,s(\Pi).$$
When $M^n$ is a minimal submanifold of a manifold with constant curvature $c$, it 
follows from \cite[Lemma 3.2]{Chen} that $\delta_M\leq \frac{1}{2}(n+1)(n-2)c$.  Thus, \emph{Chen's equality}, which is for minimal
 3-dimensional submanfolds of $\mathcal{S}^6$, is $\delta_M=2$.  
\end{dfn}

\begin{lemma}\label{Chenlemma}
A non-simple Lagrangian in $\mathcal{S}^6$ has fundamental cubic with $\SSS_3$-stabilizer at each point if and only if 
it satisfies Chen's equality.
\end{lemma}

\begin{proof}
By \cite[Lemma 3.1]{Dillen2}, $L$ is non-simple and satisfies Chen's equality if and only if there exists a non-zero 
tangent vector $\bft$ on $L$ such that $h_L(\bft,\bfv)=0$ for all tangent vectors $\bfv$ on $L$.  Letting $\bft=\bfe_1$, 
we see that the fundamental cubic of $L$ satisfying Chen's equality must be of the form 
$$a\omega_2(\omega_2^2-3\omega_3^2)+b\omega_3(3\omega_2^2-\omega_3^2).$$
Since we need only fix $\bfe_1$ in our frame, we can set $b=0$ using $\SO(2)$. 
\end{proof}


We shall see that the Lagrangians whose fundamental cubic has pointwise $\SSS_3$-symmetry 
include those associated with lower-dimensional geometries.  This leads us to prove the following result.

\begin{prop}\label{lower}
Identify $\Im\O\cong\R\oplus\C^3$ such that if $(x_1,\ldots,x_7)$ are coordinates on $\Im\O$, 
the coordinates on $\R$ and $\C^3$ are $x_1$ and $(x_2+ix_3, x_4+ix_5, x_6+ix_7)$ respectively.  
Recall that $\mathcal{S}^5\subseteq\C^3$ is endowed with a contact structure.  
\begin{itemize}
\item[\emph{(i)}] $L=\{0\}\times P\subseteq\{0\}\times\mathcal{S}^5\subseteq\mathcal{S}^6$ is Lagrangian 
if and only if $P$ is the link in $\mathcal{S}^5$ of a complex 2-dimensional cone in $\C^3$.
\item[\emph{(ii)}] $L=\{(\cos t,p\sin t)\in\R\oplus\C^3:
p\in \Sigma\subseteq\mathcal{S}^5,\,t\in(0,\pi)\}\subseteq\mathcal{S}^6$ is Lagrangian if and only if 
$\Sigma$ is a minimal Legendrian surface in $\mathcal{S}^5$.
\end{itemize}
\end{prop}

\begin{proof}  Use the notation of Definition \ref{asscoass} and use $\omega_0$ and $\Omega_0$ 
to denote the K\"ahler and holomorphic volume forms on $\C^3$. We can
write $\varphi_0$ on $\Im\O\cong\R\oplus\C^3$ as
\begin{equation}\label{phi0}
\varphi_0=\d x_1\w\omega_0+\Re\Omega_0.
\end{equation}
By Proposition \ref{coasscones}, $L$ is Lagrangian in $\mathcal{S}^6$ if and only if the cone $N$ on $L$ is coassociative; that is, 
satisfies $\varphi_0|_N\equiv 0$.  By \eq{phi0}, $N=\{0\}\times X$ is coassociative in $\Im\O\cong\R\oplus\C^3$ if and only
 if $X$ is a complex surface.  Part (i) follows.  
Similarly, $N=\R\times Y$ is coassociative if and only if $Y$ is special Lagrangian 
with phase $-i$ in $\C^3$ by Definition \ref{CYSLdfn}.  Since the link of a special Lagrangian cone in $\C^3$ is minimal Legendrian in 
$\mathcal{S}^5$, part (ii) is also proved.
\end{proof}

Lagrangian submanifolds in $\mathcal{S}^6$ satisfying Chen's equality are classified in \cite{Dillen2}.  
We review these results below.

\begin{ex}\textbf{(Links of complex cones)}\label{s4ex3}
Let $\bfu:\Sigma\rightarrow\C \P^2$ be a holomorphic curve.  Let $\mathcal{C}(\Sigma)$ be the circle 
bundle over $\Sigma$ induced by the Hopf fibration $\mathcal{S}^5\rightarrow\C\P^2$.  Let 
$\bfx:\mathcal{C}(\Sigma)\rightarrow \mathcal{S}^5$ be such that the following diagram commutes:
$$
\xymatrix{
\mathcal{C}(\Sigma)\ar[r]^{\bfx}\ar[d] & 
\mathcal{S}^5 \ar[d]
\\
\Sigma\ar[r]^{\bfu}& \C\P^2.}
$$
By \cite[Theorem 1]{Dillen2}, there exists a totally geodesic embedding $i:\mathcal{S}^5\rightarrow\mathcal{S}^6$ such 
that $i\circ\bfx:\mathcal{C}(\Sigma)\rightarrow\mathcal{S}^6$ is a Lagrangian immersion satisfying Chen's equality.  
Let $L_3(\bfu,\Sigma)=i\circ\bfx\big(\mathcal{C}(\Sigma)\big)$.  

By Proposition \ref{lower}, these examples are the links 
of complex cones embedded in a totally geodesic 5-sphere in $\mathcal{S}^6$.  Moreover, 
they are clearly tubes of radius 
$\frac{\pi}{2}$, in the plane bundle defined by the Hopf fibration, 
about the surface which is the embedding of $\Sigma$ in $\mathcal{S}^6$.
\end{ex}

\begin{note}
The Lagrangian $L_3(\bfu,\Sigma)$ where $\Sigma$ is a totally geodesic $\C\P^1$ is simple.
\end{note}

\begin{ex}\textbf{(Ruled tubes in the second normal bundle)}\label{s4ex4}
Let $\bfu:\Sigma\rightarrow\mathcal{S}^6$ be a non-totally geodesic pseudoholomorphic curve, let $h_{\Sigma}$ be its 
second fundamental form and let $\mathcal{U}(\Sigma)$ be its unit tangent bundle.  If $\Sigma$ 
has no totally geodesic points or branch points, we can define a map 
$\bfx:\mathcal{U}(\Sigma)\rightarrow\mathcal{S}^6$ by
$$\bfx:\bft\mapsto \bft\times\frac{h_{\Sigma}(\bft,\bft)}{\|h_{\Sigma}(\bft,\bft)\|}.$$ 
 By \cite[Theorem 2]{Dillen2}, $\bfx$ defines a (possibly branched) 
 Lagrangian immersion satisfying Chen's equality.  
  
From \eq{pholoframe} and Definition \ref{tubes}, we recognise $\bfx$ as defining a tube of radius $\frac{\pi}{2}$ 
in $N_2\Sigma$ about $\Sigma$.  Note that we can extend the definition of $\bfx$ if $\Sigma$ has isolated branch and totally geodesic points, and 
still get a Lagrangian immersion by \cite[Theorem 3]{Dillen2}.  Let $L_4(\bfu,\Sigma)$ be the Lagrangian associated with 
$\bfu:\Sigma\rightarrow\mathcal{S}^6$.
\end{ex}

From \cite[Theorems 4 \& 5]{Dillen2} and Lemma \ref{Chenlemma}, we deduce the following.

\begin{prop}\label{s4prop4}
Let $L$ be a connected, non-simple Lagrangian in $\mathcal{S}^6$ with fundamental cubic $C_L$.
\begin{itemize}\item[\emph{(i)}]
If $L$ is not linearly full in $\mathcal{S}^6$, $C_L$ has a pointwise $\SSS_3$-stabilizer 
and there exists a non-totally geodesic holomorphic curve $\bfu:\Sigma\rightarrow\C\P^2$ such that 
$L=L_3(\bfu,\Sigma)$ as given in Example \ref{s4ex3}.
\item[\emph{(ii)}] Suppose $L$ is linearly full and $C_L$ has a pointwise $\SSS_3$-stabilizer.  
There is an open dense subset $L^*$ of $L$ such that, for all $x\in L^*$, there exist an open set $U\ni x$ and 
a non-totally geodesic pseudoholomorphic curve $\bfu:\Sigma\rightarrow\mathcal{S}^6$ such that $U\cap L^* = 
L_4(\bfu,\Sigma)$ as given by Example \ref{s4ex4}.
\end{itemize}
\end{prop}

Notice that the Lagrangians in Examples \ref{s4ex3} and \ref{s4ex4} are both ruled as they are tubes of radius 
$\frac{\pi}{2}$. 
Therefore, the Lagrangians in $\mathcal{S}^6$ whose fundamental cubic has 
an $\SSS_3$-symmetry at each point are ruled.

\medskip

For comparison later, as well as for interest, we derive the structure equations for Lagrangians with pointwise 
$\SSS_3$-symmetry of their fundamental cubic.

Recall that we are working on an $\SSS_3$-bundle $\mathcal{F}$ over the open dense subset $L^*$ of $L$, so there exist functions $t_{ij}$ 
such that $\alpha_i=t_{ij}\omega_j$. 
Using \eq{secondLa}, \eq{secondLb} and \eq{secondLd},
we have that $t_{22}=t_{33}=-2-3t_{11}$, $t_{23}=-t_{32}$ and $t_{21}=t_{31}=0$.  If we let $t_1=t_{23}$, $t_2=-t_{13}$, $t_3=t_{12}$ and $t_0=t_{11}$, then:
\begin{subequations}\label{s3eqs1}
\begin{align}
\alpha_1&=t_0\omega_1+t_3\omega_2-t_2\omega_3;\\ 
\alpha_2&=-(2+3t_0)\omega_2+t_1\omega_3;\\
\alpha_3&=-t_1\omega_2-(2+3t_0)\omega_3;\\ 
\d a&=-a(t_1\omega_1+3t_2\omega_2+3t_3\omega_3).
\end{align}
\end{subequations}
It follows that:
\begin{subequations}\label{s3eqs2}
\begin{align}
\d\omega_1&=-2(1+3t_0)\omega_2\wedge\omega_3;\\
\d\omega_2&=-t_3\omega_2\wedge\omega_3-2t_0\omega_3\wedge\omega_1+t_1\omega_1\wedge\omega_2;\\
\d\omega_3&=t_2\omega_2\wedge\omega_3-t_1\omega_3\wedge\omega_1-2t_0\omega_1\wedge\omega_2.
\end{align}
\end{subequations}
Moreover, 
\eq{secondLc} and $\d(\d a)=0$ imply that there exist 
$u_1$, $u_2$, $v_1$, $v_{2}$ such that:
\begin{subequations}\label{s3eqs3}
\begin{align}
\d t_0&=-\textstyle\frac{2}{3}(1+3t_0)t_1\omega_1+u_1\omega_2+u_2\omega_3;\\
\d t_1&=(9t_0^2+6t_0-3-t_1^2)\omega_1-3u_2\omega_2+3u_1\omega_3;\\
\d t_2&=(2t_0t_3-t_1t_2-u_2)\omega_1+\big(\textstyle\frac{1}{2}(\textstyle\frac{1}{8}a^2-t_1^2-t_2^2-t_3^2-7-14t_0-15t_0^2)
+v_1\big)\omega_2\nonumber\\
&\qquad+\big(v_2-\textstyle\frac{1}{3}(1+3t_0)t_1\big)\omega_3;\\
\d t_3&=(u_1-2t_0t_2-t_3t_1)\omega_1+\big(v_2+\textstyle\frac{1}{3}(1+3t_0)t_1\big)\omega_2\nonumber\\
&\qquad+\big(\textstyle\frac{1}{2}(\textstyle\frac{1}{8}a^2-t_1^2-t_2^2-t_3^2-7-14t_0-15t_0^2)-v_1\big)\omega_3.
\end{align}
\end{subequations}

The appropriate EDS associated with these equations is involutive with last non-zero 
character $s_1=4$.  
 This is in agreement with 
Proposition \ref{s4prop4}, since this is the same local dependence as a pseudoholomorphic curve  in $\mathcal{S}^6$.
 
\medskip

To consider some important reductions of this system and for comparison later, we study these structure equations 
more thoroughly.  Recall the observations in $\S$\ref{pholo} and define:
\begin{subequations}\label{s3eqs4}
\begin{gather}
\bfu=-J\bfe_1;\label{s3eqs4a}\\
\bff_1=\frac{1}{2}(-J\bfe_2+i J\bfe_3);\qquad \bff_2=\frac{1}{2}(\bfe_2+i\bfe_3);\qquad \bff_3=\frac{1}{2}(\bfx+i\bfe_1);\\
\theta_1=\frac{1}{2}\big(3(1+t_0)+it_1\big)(\omega_2+i\omega_3);\\
\kappa_{22}=i\big((1+t_0)\omega_1+t_3\omega_2-t_2\omega_3\big);
\qquad\kappa_{33}=-2i\omega_1;\\
\kappa_{21}=k_2\theta_1=\frac{a}{4}(\omega_2+i\omega_3);\qquad \kappa_{31}=k_3\theta_1=0;\\
\kappa_{32}=k_1\theta_1=\frac{1}{2}\big((3t_0-1)+it_1\big)(\omega_2+i\omega_3).
\end{gather}
\end{subequations}
These functions and forms satisfy the structure equations \eq{pholoeq} for the adapted frame bundle of 
 a pseudoholomorphic curve in $\mathcal{S}^6$.
Notice from \eq{pholoeq} and \eq{s3eqs4} that $J\bfe_1$ is constant if and only if $\theta_1=0$, 
which is if and only if $1+t_0=t_1=0$.  
Thus, the Lagrangian lies in a totally geodesic $\mathcal{S}^5$ if and only if $(t_0,t_1)=(-1,0)$.  

If $(t_0,t_1)\neq(-1,0)$, $\theta_1$ is nowhere vanishing on some open dense set, and so $\bfu$ in \eq{s3eqs4a} 
defines a pseudoholomorphic curve $\Sigma$ in $\mathcal{S}^6$.  Moreover, as $\kappa_{31}=0$ on $\Sigma$, the unitary 
frame $\{\bff_1,\bff_2,\bff_3\}$ over $\Sigma$ is adapted so that $\bff_2$ spans $N_1\Sigma$ and $\bff_3$ spans $N_2\Sigma$.  
Since $\bfx=\Re\bff_3$ and $\bfe_1=\Im\bff_3$, the Lagrangian defined by $\bfx$ is a tube of radius $\frac{\pi}{2}$ 
in $N_2\Sigma$ about $\Sigma$ as claimed in Proposition \ref{s4prop4}(ii).
\medskip

Suppose that $t_1=0$.  Then either $t_0=-1$ or $t_0=\frac{1}{3}$ and in each case the reduced EDS is involutive with $s_1=2$. 
As noted above, the system for $(t_0,t_1)=(-1,0)$ describes the Lagrangians given in Example \ref{s4ex3}.
For $(t_0,t_1)=(\frac{1}{3},0)$, the Lagrangians must necessarily be locally of the form in Example \ref{s4ex4}
by Proposition \ref{s4prop4}.  From \eq{s3eqs4}, $(t_0,t_1)=(\frac{1}{3},0)$ if and only if the torsion 
$k_1=0$, so the Lagrangians are tubes about \emph{null-torsion} pseudoholomorphic curves.

Suppose we consider the reduced system where $a$ is constant.  This forces $t_i=0$ for $i=1,2,3$ 
and either $(a,t_0)=(8,-1)$ or  $(a,t_0)=(\frac{8}{3}\sqrt{15},\frac{1}{3})$.

\begin{ex}\textbf{({\boldmath $\SO(3)$}-orbits 2)}\label{aconstantex1} 
If $(a,t_0)=(8,-1)$, our comments above show that the Lagrangian lies in a totally geodesic $\mathcal{S}^5$, and: 
\begin{align*}
\d\omega_1=4\omega_2\w\omega_3;\qquad\d\omega_2&=2\omega_3\w\omega_1;\qquad\d\omega_3=2\omega_1\w\omega_2.
\end{align*}
We observe 
that the equations for $\omega_2$ and $\omega_3$ define a constant curvature 2-sphere which cannot 
be totally geodesic, otherwise the Lagrangian would be simple.  Thus, the underlying holomorphic curve   
is the degree 2 $\C\P^1$ in $\C\P^2$. 
Further, it is a homogeneous
submanifold of $\mathcal{S}^6$, so we deduce from \cite[Theorems 4.2 \& 4.4]{Mashimo} that it is invariant under 
an $\SO(3)$ action on $\Im\O\cong\R\oplus\C^3$, where $\SO(3)$ acts trivially on $\R$ and as the standard 
(real) $\SO(3)$ action on $\C^3$.  Hence, up to rigid motion, the Lagrangian is  
$$\{(0,z_1,z_2,z_3)\in\R\oplus\C^3\,:\, z_1^2+z_2^2+z_3^2=0\}\cap\mathcal{S}^6,$$
which is the Hopf lift of the Veronese curve 
$\bfu:\C\P^1(2)\rightarrow\C\P^2$.
\end{ex}

\begin{ex}\textbf{({\boldmath $\SO(3)$}-orbits 3)}\label{aconstantex2} 
If $(a,t_0)=(\frac{8}{3}\sqrt{15},\frac{1}{3})$: 
\begin{align*}
\d\omega_1=-4\omega_2\w\omega_3;\qquad
\d\omega_2=-\frac{2}{3}\,\omega_3\w\omega_1;\qquad
\d\omega_3=-\frac{2}{3}\,\omega_1\w\omega_2.
\end{align*}
Again this is a homogeneous submanifold of $\mathcal{S}^6$ and so, by process of elimination, we can deduce 
from \cite[Theorems 4.3 \& 4.4]{Mashimo} that it is (up to $\GGG_2$ transformation) the orbit through $\eps_6\in\Im\O$ 
of the $\SO(3)$ action given in Example \ref{irractiondfn}.  Equivalently, it is the $\SO(3)$-orbit of the 
cubic $\frac{1}{2}y(y^2-3z^2)$ in $\mathcal{H}^3(\R^3)$, which is clearly 
diffeomorphic to $\SO(3)/\SSS_3$ by Table \ref{symtable}.  Moreover, the structure equations for $\omega_2$ and $
\omega_3$ define 
a constant curvature 2-sphere and, since the Lagrangian is linearly full, it must once again 
be the Bor\r{u}vka sphere.  Hence this example is a tube of radius $\frac{\pi}{2}$ in the second normal bundle about 
$\mathcal{S}^2(\frac{1}{6})$. 
\end{ex}

\noindent The author, in 
\cite{Lotay}, gave a method for producing examples of coassociative 4-folds in $\Im\O$ with symmetries.  We can 
apply this method to the $\SO(3)$-action given in Example \ref{irractiondfn}, and conical solutions will give rise 
to the Lagrangian $\SO(3)$-orbits by Proposition \ref{coasscones}.  Altogether, we can derive a system of first-order
 ordinary differential equations 
 whose solutions define the homogeneous Lagrangians we are interested in.  Therefore,
in principle, this system can simply be integrated to produce an explicit description of the Lagrangian given in Example 
\ref{aconstantex2}, though the author has been unable to do this.

\begin{note} Examples \ref{simple}, \ref{s4ex2}, \ref{s4ex1}, \ref{aconstantex1} and \ref{aconstantex2}
classify the homogeneous Lagrangian submanifolds of $\mathcal{S}^6$ up to $\GGG_2$ transformation, as studied in \cite{Mashimo}.
\end{note}

\begin{ex}\textbf{(Products 1)}\label{prodsex1}
Setting $t_0=-\frac{1}{3}$, we see immediately that $\d\omega_1=0$ and that the reduced EDS is still involutive but now with $s_1=2$.
Since $t_0\neq -1$, the corresponding Lagrangians must locally be of the form $L_4(\bfu,\Sigma)$ 
as in Example \ref{s4ex4} by Proposition \ref{s4prop4}.  Moreover, these Lagrangians are locally \emph{products} 
$\mathcal{S}^1\times P$ for some surface $P$, which is equivalent to saying that $N_2\Sigma$ is 
trivial.  
 Furthermore, by \eq{s3eqs4}, we see that $t_0=-\frac{1}{3}$ if and only if the torsion $k_1$ of $\Sigma$ satisfies 
$|k_1|=1$.  As observed in $\S$\ref{pholo}, this occurs if and only if 
$\Sigma$ lies linearly full in a totally geodesic $\mathcal{S}^5$ and so is a minimal Legendrian surface.
\end{ex}

At this point we make an aside concerning \emph{austere} Lagrangians in $\mathcal{S}^6$.

\begin{dfn}\label{austere}
Let $L$ be a 3-dimensional submanifold of $\mathcal{S}^6$ and let $h_L$ be its second fundamental form. For each $p\in L$,
let $\{\bfe_1(p),\bfe_2(p),\bfe_3(p)\}$ and $\{\bfe_1^{\perp}(p),\bfe_2^{\perp}(p),\bfe_3^{\perp}(p)\}$ be orthonormal
 bases for $T_pL$ and $N_pL$ respectively and let $2\omega_i$ be the dual 1-form to $\bfe_i$. We can locally write 
$h_L$ using summation notation as:
$$h_L=4h_{ijk}\bfe^{\perp}_i\otimes\omega_j\omega_k$$
for some tensor of functions $h_{ijk}$ satisfying $h_{ijk}=h_{ikj}$.  Let $q_i=4h_{ijk}\omega_j\omega_k$ in 
summation notation.  We say that $L$ is \emph{austere} if, for all $i$, the set of eigenvalues of $q_i$ is of the form 
$\{0,\pm\lambda_i\}$ for some $\lambda_i$.  
\end{dfn}

\noindent Austere submanifolds were introduced in \cite{HarLaw} in the discussion 
of special Lagrangian submanifolds.  Notice that austere 3-folds in $\mathcal{S}^6$ are minimal.  A 
complete classification of austere 3-folds in $\mathcal{S}^6$ is not known, but steps in this direction are 
taken in \cite{Dajczer} and \cite{Ikawa}.  However, we are able to show the following.

\begin{prop}
A Lagrangian in $\mathcal{S}^6$ is austere if and only if its fundamental cubic is either zero or has 
$\AAA_4$ or $\SSS_3$-stabilizer at each point.  Thus, connected austere Lagrangians in $\mathcal{S}^6$ are either simple or given by 
Proposition \ref{s4prop3} or Proposition \ref{s4prop4}.
\end{prop}

\begin{proof}  Let $L$ be a Lagrangian in $\mathcal{S}^6$ and let $C_L$ be its fundamental cubic.  Clearly a simple 
Lagrangian is austere, so assume $L$ is non-simple, so that $C_L\neq 0$.
 Recall 
the cubic $C(r,s,a,b)$ defined in Lemma \ref{rcubic}.  
By a result of Vrancken \cite{Vrancken3}, there exists a frame on $L$ such that $C_L=C(r,s,a,b)$ for 
some functions $r,s,a,b$.  Moreover, if $s=0$ we can choose $b=0$.  (The key difference between the result in \cite{Vrancken3} and Lemma \ref{rcubic} is 
that $r$ has to be constant if $L$ is ruled or quasi-ruled.) 
 
We can calculate the quadratic forms $q_i$ as in Definition \ref{austere} as follows:
\begin{align*}
q_1&=2r\omega_1^2-(r-s)\omega_2^2-(r+s)\omega_3^2;\\
q_2&=-2(r-s)\omega_1\omega_2+a(\omega_2^2-\omega_3^2)+2b\omega_2\omega_3;\\
q_3&=-2(r+s)\omega_1\omega_2+b(\omega_2^2-\omega_3^2)-2a\omega_2\omega_3.
\end{align*}
Thus, $L$ is austere if and only if $r(r^2-s^2)$, $(r-s)^2a$ and $(r+s)^2b$ are all zero.  Therefore, 
$C_L$ is, up to a choice of frame, $C(0,s,0,0)$, $C(0,0,a,0)$ or $C(r,r,a,0)$. 
By Table \ref{symtable} and the remarks preceding it, these cubics have pointwise $\AAA_4$ or $\SSS_3$-stabilizers.
\end{proof}

\subsection[Z3]{{\boldmath $\Z_3$}}

Suppose that $C_L\neq0$ has a $\Z_3$-stabilizer at each point.  Therefore, there exist an
open dense subset $L^*$ of $L$ and functions $r,a:L^*\rightarrow\R^+$, with $a\neq r\sqrt{2}$, 
such that
$C_L=r\omega_1\big(2\omega_1^2-3\omega_2^2-3\omega_3^2\big)+a\omega_2(\omega_2^2-3\omega_3^2)$
defines a $\Z_3$-subbundle $\mathcal{F}$ of the adapted frame bundle over $L^*$.  

Calculating the Ricci tensor on $L^*$, using Definition \ref{Kdfn} and Proposition \ref{repeat}, 
we find that it has repeated eigenvalues, so $L$ is quasi-Einstein.  Furthermore, Table \ref{fibretable} shows 
that $L$ is non-simple and quasi-Einstein if and only if $C_L$ has a pointwise $\SO(2)$, $\AAA_4$, $\SSS_3$ or 
$\Z_3$-stabilizer.  As mentioned before, quasi-Einstein Lagrangians in $\mathcal{S}^6$ are classified in 
\cite{Deszcz}, so we can give the remaining examples.

\begin{ex}\textbf{(Quasi-ruled tubes in the second normal bundle)}\label{s4ex5}
Let $\bfu:\Sigma\rightarrow\mathcal{S}^6$ be a null-torsion pseudoholomorphic curve with 
no totally geodesic points. Let $h_{\Sigma}$ be the second fundamental form of $\Sigma$ and let 
$\mathcal{U}(\Sigma)$ be its unit tangent bundle.  Define 
$\bfx:\mathcal{U}(\Sigma)\rightarrow\mathcal{S}^6$ by 
$$\bfx:\bft\mapsto \frac{\sqrt{5}}{3}\,\bfu+
\frac{2}{3}\,\bft\times\frac{h_{\Sigma}(\bft,\bft)}{\|h_{\Sigma}(\bft,\bft)\|}.$$
By \cite[Theorem 1]{Deszcz}, $\bfx$ is a Lagrangian immersion. If $\Sigma$ has isolated totally geodesic points, 
$\bfx$ defines an immersion on an open dense subset of $\mathcal{U}(\Sigma)$ and its image $L_5(\bfu,\Sigma)$ is Lagrangian in 
$\mathcal{S}^6$.  
 Moreover, $L_5(\bfu,\Sigma)$ is quasi-Einstein and does not satisfy Chen's equality.  We deduce from Table \ref{fibretable}, 
Lemma \ref{Chenlemma}, and Propositions \ref{s4prop2}, \ref{s4prop3} and \ref{s4prop4} that 
$L_5(\bfu,\Sigma)$ has fundamental cubic with $\Z_3$-stabilizer as long as $\Sigma\neq\mathcal{S}^2(\frac{1}{6})$.

From Proposition \ref{tubes} and \eq{pholoframe}, we quickly see that $L_5(\bfu,\Sigma)$ is a tube of radius 
$\sin^{-1}(\frac{2}{3})$ in $N_2\Sigma$ about $\Sigma$, and thus is quasi-ruled.
\end{ex} 

\begin{remark} The Lagrangian $L_2$ given in Example \ref{s4ex1}, whose fundamental cubic has $\AAA_4$-stabilizer, 
is $L_5(\bfu,\Sigma)$ where $\Sigma$ is the Bor\r{u}vka sphere in $\mathcal{S}^6$.
\end{remark}

Combining \cite[Theorems 1 \& 2]{Deszcz} and our observations thus far, we get the next result.  
 Notice that the constant curvature null-torsion pseudoholomorphic curves are 
totally geodesic 2-spheres and $\mathcal{S}^2(\frac{1}{6})$.   

\begin{prop}\label{s4prop6}
A connected, non-simple, Lagrangian $L$ in $\mathcal{S}^6$ has a fundamental 
cubic with a $\Z_3$-stabilizer at each point if and only if there exists an open dense subset $L^*$ of $L$ such 
that, for every point $x\in L^*$, there exist an open set $U\ni x$ and a null-torsion, non-constant curvature,
 pseudoholomorphic curve $\bfu:\Sigma\rightarrow\mathcal{S}^6$ with 
$U\cap L^*=L_5(\bfu,\Sigma)$ as given in Example \ref{s4ex5}.
\end{prop}

Now, for possible interest, we record the structure equations. 
Recall that we have a $\Z_3$-subbundle $\mathcal{F}$ of the adapted frame bundle over 
an open dense subset $L^*$ of $L$. 
 We deduce from \eq{secondL} that $r=2\sqrt{5}$, so we can take $L^*=L$ and
 $C_L=2\sqrt{5}\,\omega_1\big(2\omega_1^2-3\omega_2^2-3\omega_3^2\big)+a\omega_2(\omega_2^2-3\omega_3^2)$.  
 Further, there exist functions $t_2,t_3$ on $\mathcal{F}$ such that:
\begin{subequations}\label{z3eqs1}
\begin{align}
\alpha_1=-\textstyle\frac{1}{2}\omega_1+t_3\omega_2-t_2\omega_3;\qquad
\alpha_2=-\textstyle\frac{1}{2}\omega_2;\qquad
\alpha_3=-\textstyle\frac{1}{2}\omega_3;
\end{align}
and
\begin{align}
\d a=-3a(t_2\omega_2+t_3\omega_3).
\end{align}
\end{subequations}
We may therefore write down the structure equations:
\begin{subequations}\label{z3eqs2}
\begin{align}
\d\omega_1&=\omega_2\wedge\omega_3;\\
\d\omega_2&=-t_3\omega_2\wedge\omega_3+\omega_3\wedge\omega_1;\\
\d\omega_3&=t_2\omega_2\wedge\omega_3+\omega_1\wedge\omega_2.
\end{align}
\end{subequations}
Furthermore, 
there exist functions $u_2$ and $u_3$ such that:
\begin{subequations}\label{z3eqs3}
\begin{align}
\d t_2&=-t_3\omega_1+\left(\textstyle\frac{1}{16}a^2+u_2-\textstyle\frac{1}{2}(t_2^2+t_3^2+5)\right)\omega_2+u_3\omega_3;\\
\d t_3&=t_2\omega_1+u_3\omega_2+\left(\textstyle\frac{1}{16}a^2-u_2-\textstyle\frac{1}{2}(t_2^2+t_3^2+5)\right)\omega_3.
\end{align}
\end{subequations}


\begin{remark}
We can interpret a result of Fox \cite[Theorem 9.3]{Fox2} as follows: a coassociative cone $N_0$ on a
 Lagrangian in $\mathcal{S}^6$ whose fundamental cubic has $\Z_3$-stabilizer is the limit, as $t\rightarrow 0$, of 
a family of nonsingular coassociative 4-folds $N_t$ in $\Im\O$, 
which are \emph{asymptotically conical} to $N_0$ at infinity.  
\end{remark}

\subsection[Z2]{{\boldmath $\Z_2$}}

In similar second order family 
studies to the type considered here, the $\Z_2$ case is typically the most complicated and 
hardest to classify.  However, for Lagrangians in $\mathcal{S}^6$, it could not be easier \cite[Theorem 2]{Vrancken}.

\begin{prop}
The only connected Lagrangian submanifolds of $\mathcal{S}^6$ whose fundamental
 cubic has a $\Z_2$-stabilizer at each point are simple.
\end{prop}

\subsection[Summary]{Summary}

We have shown that a Lagrangian $L$ in $\mathcal{S}^6$ whose fundamental cubic has a non-trivial stabilizer in $\SO(3)$ 
at each point is either ruled or quasi-ruled.  Moreover, $L$ is
 a Hopf lift to $\mathcal{S}^5\subseteq\mathcal{S}^6$ of a holomorphic curve in $\C\P^2$, or 
given locally as a tube about a pseudoholomorphic curve $\Sigma$.  
 Further, the tube is in $N_2\Sigma$ if $\Sigma$ is non-totally geodesic. 
Thus, we can summarise our results by associating to each non-trivial stabilizer a holomorphic curve, or a pseudoholomorphic curve and 
a tube radius.  This is the content of Table \ref{summary} below.

\begin{table}[h]
\begin{center}
\begin{tabular}{|c|c|c|c|}
\hline
Stabilizer & Holomorphic curve & Pseudoholomorphic curve & Tube radius 
\\
\hline
$\SO(3)$ & Totally geodesic 
& Totally geodesic 
& $\frac{\pi}{2}$ 
\\
$\SO(2)$ & & Totally geodesic 
& $\sin^{-1}(\frac{2}{3})$
\\
$\AAA_4$ & & Null-torsion 
$\mathcal{S}^2(\frac{1}{6})$ & $\sin^{-1}(\frac{2}{3})$
\\
$\SSS_3$ & 
Non-totally geodesic 
& 
Non-totally geodesic 
& $\frac{\pi}{2}$
\\
$\Z_3$ & & Null-torsion not $\mathcal{S}^2(\frac{1}{6})$ & $\sin^{-1}(\frac{2}{3})$
\\
$\Z_2$ & Totally geodesic 
& 
Totally geodesic 
& $\frac{\pi}{2}$
\\
\hline
\end{tabular}
\caption{Summary of examples as Hopf lifts of holomorphic curves in $\C\P^2$ and 
tubes about pseudoholomorphic curves in $\mathcal{S}^6$}\label{summary}
\end{center}
\end{table}

\begin{remark}
For stabilizer $\GGG\neq\SSS_3$ or $\Z_3$, the examples in Table \ref{summary} are rigid; 
i.e.~they are unique up to $\GGG_2$ transformations of $\mathcal{S}^6$. However, the $\SSS_3$ and $\Z_3$ 
examples have non-trivial deformations given by deformations of the underlying curve.
\end{remark}

\section[The ruled Lagrangian family]{The ruled Lagrangian family}\label{s7}

Before we discuss the general ruled family, we prove the following result.

\begin{prop}\label{ruledsimple}
Any connected Lagrangian in $\mathcal{S}^6$ with two distinct 1-rulings is simple.
\end{prop}

\begin{proof} If a connected Lagrangian $L$ in $\mathcal{S}^6$ has two distinct 1-rulings then the cone $N$ on $L$ in $\Im\O$ is  
coassociative with two distinct 2-rulings by Lemma \ref{r2rlemma}. In $\O\cong\R^8$ there are calibrated 4-dimensional 
submanifolds called \emph{Cayley} 4-folds.  If we embed $N$ in $\O=\R\oplus\Im\O$ as $C=\{0\}\times N$, 
then $C$ is a 2-ruled Cayley cone by \cite[Proposition 2.11]{Lotay2R}.  Moreover, $C$ has two distinct 2-rulings and so, by 
 \cite[Proposition 4.4]{Lotay2R}, must be a 4-plane.  Thus, $N$ is a 4-plane and $L$ is simple as claimed.
%
%
\end{proof}

\begin{remark} This is the analogue of \cite[Theorem 6 part 4]{Bryant2} and we could have proved it in an 
analogous manner.  The key points are that a Lagrangian $L$ has two distinct 1-rulings if and only if $C_L$ has 
$\AAA_4$ or $\Z_2$-stabilizer, and that the only ruled Lagrangians such that $C_L$ has $\AAA_4$ or $\Z_2$-stabilizer are simple by 
Table \ref{summary}.
\end{remark}

Let $L$ be a Lagrangian in $\mathcal{S}^6$ ruled by geodesic circles. By Lemma \ref{rcubic}, 
using the notation there, we can choose a frame on $L$ such that the fundamental cubic $C_L=C(0,s,a,b)$ for some  
functions $s,a,b$.  
We are interested in the possibility of ruled Lagrangian submanifolds $L$ for which $C_L$ does not have a pointwise symmetry, so we make this assumption.

As stated in Lemma \ref{rcubic}, if $s=0$ we can choose $b=0$.  However, $C(0,0,a,0)$ has at least an 
$\SSS_3$-stabilizer at each point by Table \ref{symtable}. We also notice from Table \ref{symtable} 
that $C(0,s,0,0)$ has a pointwise symmetry.  Thus, 
we assume that $s$ and $a^2+b^2$ are both non-zero on some open dense subset $L^*$ of $L$.

Using \eq{secondL} we calculate:
\begin{subequations}\label{ruledeqs1}
\begin{align}
\alpha_1=& \,t_0\omega_1+t_3\omega_2-t_2\omega_3;\displaybreak[0]\\
\alpha_2=&-(1+t_0)\omega_2+t_1\omega_3;\displaybreak[0]\\
\alpha_3=&-t_1\omega_2-(1+t_0)\omega_3;\displaybreak[0]\\
\d s=&-2s(t_1\omega_1+t_2\omega_2+t_3\omega_3);\displaybreak[0]\\
\d a=&-\big(2st_2+at_1+b(1+2t_0)\big)\omega_1+(c_1-st_1-\textstyle\frac{3}{2}at_2-\textstyle\frac{3}{2}bt_3)\omega_2\nonumber\\
&+\big(c_2-(1+2t_0)s-\textstyle\frac{3}{2}at_3-\textstyle\frac{3}{2}bt_2\big)\omega_3;\displaybreak[0]\\
\d b=&-\big(2st_3-a(1+2t_0)+bt_1\big)\omega_1+\big(c_2+(1+2t_0)s+\textstyle\frac{3}{2}at_3+\textstyle\frac{3}{2}bt_2\big)\omega_2\nonumber\\
&-(c_1+st_1+\textstyle\frac{3}{2}at_2+\textstyle\frac{3}{2}bt_3)\omega_3;
\end{align}
\end{subequations}
for some functions $t_0$, $t_1$, $t_2$, $t_3$, $c_1$, $c_2$.  Thus, the structure equations are:
\begin{subequations}\label{ruledeqs2}
\begin{align}
\d\omega_1&=-2t_0\omega_2\w\omega_3;\\
\d\omega_2&=-t_3\omega_2\w\omega_3+\omega_3\w\omega_1+t_1\omega_1\w\omega_2;\\
\d\omega_3&=t_2\omega_2\w\omega_3-t_1\omega_3\w\omega_1+\omega_1\w\omega_2.
\end{align}
\end{subequations}
Moreover, there exist functions $u_1,u_2,u_3,u_4$ such that:
\begin{subequations}\label{ruledeqs3}
\begin{align}
\d t_0&=-2t_0t_1\omega_1+(u_1-\textstyle\frac{1}{16}sb)\omega_2+(\textstyle\frac{1}{16}sa+u_2)\omega_3;\\
\d t_1&=\big(\textstyle\frac{1}{16}s^2+t_0^2-t_1^2-4\big)\omega_1+(\textstyle\frac{1}{16}sa-u_2)\omega_2+(u_1+\textstyle\frac{1}{16}sb)\omega_3;\\
\d t_2&=\big(\textstyle\frac{1}{16}sa-t_3-t_1t_2-u_2\big)\omega_1+\big(u_4-t_0t_1\big)\omega_3\nonumber\\
&+\big(\textstyle\frac{1}{2}(\textstyle\frac{1}{16}s^2+\textstyle\frac{1}{8}a^2+\textstyle\frac{1}{8}b^2-t_0(2+3t_0)-t_1^2-t_2^2-t_3^2-4)+u_3\big)\omega_2;\\
\d t_3&=\big(\textstyle\frac{1}{16}sb+t_2-t_3t_1+u_1\big)\omega_1+\big(u_4+t_0t_1\big)\omega_2\nonumber\\
&+\big(\textstyle\frac{1}{2}(\textstyle\frac{1}{16}s^2+\textstyle\frac{1}{8}a^2+\textstyle\frac{1}{8}b^2-t_0(2+3t_0)-t_1^2-t_2^2-t_3^2-4)-u_3\big
)\omega_3.
\end{align}
\end{subequations}

Setting up the appropriate EDS here, we find that it is involutive with last non-zero Cartan character $s_1=6$.  
Thus, ruled Lagrangian submanifolds of $\mathcal{S}^6$ depend locally on 6 functions of 1 variable.  
The largest ruled family we have seen so far (Example \ref{s4ex4}) depends only on 4 functions of 1 variable locally,
 so there must be another family describing the general ruled Lagrangians.


We shall describe a family of ruled Lagrangians in $\mathcal{S}^6$ 
with the ``right'' local dependence on functions of one variable, then prove that this family provides a 
local classification for the generic ruled Lagrangian.  We start with a definition.

\begin{dfn}\label{Xdfn}
Let $\bfu:\Sigma\rightarrow\mathcal{S}^6$ be a pseudoholomorphic curve and use the notation of $\S$\ref{pholo}.  
In particular, denote a special unitary frame for $\bfu^*(T\mathcal{S}^6)$ by $\{\bff_1,\bff_2,\bff_3\}$ such that 
$\bff_1$ spans $T^{1,0}\Sigma$ and $\{\bff_2,\bff_3\}$ is a basis for $N\Sigma$.  
 Let $\mathcal{B}(\Sigma)$ 
be the $\U(2)$-bundle of such frames $\{\bff_1,\bff_2,\bff_3\}$ over $\Sigma$.  

We define two $\U(1)$ actions on $\mathcal{B}(\Sigma)$.  The first, $\U(1)_l$, is the action which fixes $\bff_3$: 
$$(\bff_1,\bff_2,\bff_3)\longmapsto (e^{it}\bff_1,e^{-it}\bff_2,\bff_3).$$
The second, $\U(1)_r$, is the rotation of $\bff_3$:
$$(\bff_1,\bff_2,\bff_3)\longmapsto (e^{-it}\bff_1,e^{-it}\bff_2,e^{2it}\bff_3).$$
Let $\mathcal{Q}(\Sigma)=\mathcal{B}(\Sigma)/\U(1)_l$ and let $\mathcal{X}(\Sigma)=\mathcal{Q}(\Sigma)/\U(1)_r$. 
Note that we have a projection $\pi_\mathcal{X}:\mathcal{Q}(\Sigma)\rightarrow \mathcal{X}(\Sigma)$ 
whose fibres are circles.  
\end{dfn}

Clearly, surfaces in the 4-manifold $\mathcal{X}(\Sigma)$ lift to 3-dimensional submanifolds of $\mathcal{Q}(\Sigma)$ via
 $\pi_{\mathcal{X}}$.  Moreover, these 3-folds can be thought of as tubes in $N\Sigma$ about $\Sigma$.  Hence, 
the tubes of this type which are Lagrangian are equivalent to distinguished surfaces in $\mathcal{X}(\Sigma)$.  
This motivates the following key result. 

\begin{thm}\label{Xcomplex}  Let $\bfu:\Sigma\rightarrow\mathcal{S}^6$ be a non-totally geodesic pseudoholomorphic curve  
and use the notation of Definition \ref{Xdfn}.
 There is an integrable complex structure $I$ 
on $\mathcal{X}(\Sigma)$ such that $\mathcal{X}(\Sigma)$ is a holomorphic $\C\P^1$-bundle over $\Sigma$, and a 
(real) surface $\Gamma$ in $\mathcal{X}(\Sigma)$ is a holomorphic curve 
if and only if the image of the map $\bfx:\pi_\mathcal{X}^{-1}(\Gamma)\rightarrow\mathcal{S}^6$ given by 
$\bfx=\bff_3+\bar{\bff}_3$ is a ruled Lagrangian in $\mathcal{S}^6$.
\end{thm}

\begin{proof}
Define $\bfx:\mathcal{Q}(\Sigma)\rightarrow\mathcal{S}^6$ by $\bfx=\bff_3+\bar{\bff}_3$.  This map is certainly 
well-defined on $\mathcal{Q}(\Sigma)$ since it is defined on $\mathcal{B}(\Sigma)$ and the action of $\U(1)_l$ 
fixes $\bff_3$.  
Using the structure equations \eq{pholoeq} for the adapted frame bundle over $\Sigma$, 
we see that:
$$\d\bfx=-\bar{k}_3\bff_1\bar{\theta}_1-k_3\bar{\bff}_1\theta_1-\bff_2(\bar{\kappa}_{32}-\bar{\theta}_1)-\bar{\bff}_2(\kappa_{32}
-\theta_1)+(\bff_3-\bar{\bff}_3)\kappa_{33},$$
where we remind the reader that $\theta_1$ is the dual 1-form to $\bff_1$, $\kappa$ is a $3\times 3$ 
traceless skew-Hermitian matrix and $k_3$ is a holomorphic function.
Therefore, from \eq{times}, 
$$\bfx\times\d\bfx=-i\bfu\kappa_{33}+\bff_1(\kappa_{32}-\theta_1)+\bar{\bff}_1(\bar{\kappa}_{32}-\bar{\theta}_1)
-k_3\bff_2\theta_1-\bar{k}_3\bar{\bff}_2\bar{\theta}_1.$$
Recall Definition \ref{6spheredfn}.  In particular, the almost complex structure $J$ on $\mathcal{S}^6$ is given by cross product
 with the position vector on $\mathcal{S}^6$ and the almost symplectic form $\omega$ on $\mathcal{S}^6$ is given by
 $\omega(\bfu,\bfv)=J\bfu\,\textbf{.}\,\bfv$. Thus the pull-back of $\omega$ by $\bfx$ is given by
$\bfx^*(\omega)=(\bfx\times\d\bfx)\,\textbf{.}\,\d\bfx$, and we deduce that
\begin{equation}\label{chomega}
\check{\omega}=\frac{1}{4}\,\bfx^*(\omega)=\Re(\kappa_{32}\w k_3\theta_1).
\end{equation}
Again using \eq{pholoeq}, we calculate that 
\begin{equation}\label{chUpsilon}
\d\check{\omega}=-3i\kappa_{33}\w\check{\Upsilon}\quad\text{where}\quad\check{\Upsilon}=\Im(\kappa_{32}\w k_3\theta_1).
\end{equation}
Therefore, we are lead to define the 2-form $\check{\Omega}$ on $\mathcal{Q}(\Sigma)$ by
\begin{equation}\label{chOmega}
\check{\Omega}=\check{\omega}+i\check{\Upsilon}=\kappa_{32}\w k_3\theta_1 .
\end{equation}
By \eq{pholoeq}, it is clear that
\begin{equation}\label{cxstructure}
\d\check{\Omega}=-3\kappa_{33}\w\check{\Omega}.
\end{equation}
Equation \eq{cxstructure} shows that $\check{\Omega}$  pushes down 
to $\mathcal{X}(\Sigma)$.  We also see, from \eq{chomega}-\eq{chOmega}, that
$\bfx:L^3\subseteq\mathcal{Q}(\Sigma)\rightarrow\mathcal{S}^6$ is a Lagrangian immersion if and only if 
$\check{\Omega}|_L\equiv 0$.  

Since $\Sigma$ is non-totally geodesic we can assume that $k_3$ only has 
isolated zeros. Further, as we have not adapted frames so that $\bff_2$ and $\bff_3$ span $N_1\Sigma$ and $N_2\Sigma$, 
 $\kappa_{32}$ is a non-zero 1-form independent of $\theta_1$.  
Thus, $\check{\Omega}$ is a definite 2-form on an open dense subset of $\mathcal{X}(\Sigma)$. Hence, by \eq{cxstructure}, 
$\check{\Omega}$ defines an integrable complex structure $I$ on $\mathcal{X}(\Sigma)$. Notice, trivially, 
that the $(1,0)$-form $\theta_1$ on $\Sigma$ pulls back to $\mathcal{X}(\Sigma)$ to be a linear combination of $\kappa_{32}$ 
and $k_3\theta_1$, away from the zeroes of $k_3$.  Thus, the projection from $\mathcal{X}(\Sigma)$ to $\Sigma$ is holomorphic 
and $\mathcal{X}(\Sigma)$ is a $\C\P^1$-bundle over $\Sigma$.  Moreover, $I$ has the property that 
a real surface $\Gamma$ in $\mathcal{X}(\Sigma)$ is a holomorphic curve 
if and only if  
$L=\pi_\mathcal{X}^{-1}(\Gamma)\subseteq\mathcal{Q}(\Sigma)$ 
satisfies $\check{\Omega}|_L\equiv0$.  Since $L$ is clearly ruled, the result follows.
\end{proof}

\noindent If $\Sigma$ is totally geodesic, the proof of Theorem \ref{Xcomplex} shows that every lift of $\Sigma$ to 
$\mathcal{X}(\Sigma)$ defines a ruled Lagrangian.  However, these Lagrangians must be invariant under the $\SO(4)$ subgroup 
of $\GGG_2$ preserving $\Sigma$ and so are simple.

\begin{remarks}
Since $\mathcal{X}(\Sigma)$ is a $\C\P^1$-bundle over a Riemann surface $\Sigma$, it is a \emph{ruled} 
complex surface and so is \emph{algebraic} by Kodaira's classification of complex surfaces.  Therefore, holomorphic curves 
in $\mathcal{X}(\Sigma)$ are well-understood using techniques in algebraic geometry.  Moreover, if $\Sigma$ is non-compact, 
$\mathcal{X}(\Sigma)$ is trivial and so holomorphic curves in $\mathcal{X}(\Sigma)$ are equivalent to holomorphic functions 
from $\Sigma$ to $\C\P^1$.
\end{remarks}

We now present the most general family of ruled Lagrangians in $\mathcal{S}^6$.

\begin{ex}\textbf{(The general ruled family)}\label{s7ex1}
Let $\bfu:\Sigma\rightarrow\mathcal{S}^6$ be a non-totally geodesic pseudoholomorphic curve 
 and use the notation of Definition \ref{Xdfn}.  By Theorem \ref{Xcomplex}, the bundle $\mathcal{X}(\Sigma)$  
is endowed with an integrable complex structure.  
Let $\Gamma$ be a holomorphic curve in $\mathcal{X}(\Sigma)$ and 
 define $\bfx:\pi_{\mathcal{X}}^{-1}(\Gamma)\rightarrow\mathcal{S}^6$ by $\bfx=\bff_3+\bar{\bff}_3$.  
By Theorem \ref{Xcomplex}, the image $L_6(\bfu,\Sigma,\Gamma)$ of $\bfx$ is a ruled Lagrangian.  

Let $\Pi$ be the 2-plane bundle over $\Sigma$ defined by $2i\bar{\bff}_3\w\bff_3$.  
By Definition \ref{tubes}, $L_6(\bfu,\Sigma,\Gamma)$ is a tube of radius $\frac{\pi}{2}$ in $\Pi$ 
about $\Sigma$.
Furthermore, these examples depend locally on 6 functions 
of 1 variable since our data consists of a pseudoholomorphic curve in $\mathcal{S}^6$ and a holomorphic curve 
in a complex 2-manifold.
\end{ex}

Since $\mathcal{X}(\Sigma)$ is a subbundle of the frame bundle over $\Sigma$, the structure equations on it 
are given by the $\GGG_2$ structure equations \eq{pholoeqgen} for some vector of 1-forms 
$\theta=(\theta_1,\theta_2,\theta_3)^{\rm T}$ and some traceless $3\times 3$ skew-Hermitian matrix $\kappa$.  Moreover, 
 from Definition \ref{Xdfn}, we see that we can adapt frames on $\Sigma$ so that $\theta_3=0$ on $\mathcal{X}(\Sigma)$.
From the proof of Theorem \ref{Xcomplex}, a surface $\Gamma$ in $\mathcal{X}(\Sigma)$ is a holomorphic curve if and only if 
$\kappa_{31}\w\kappa_{32}$ vanishes on $\Gamma$ (recalling that $\kappa_{31}=k_3\theta_1$ on $\Sigma$).  Thus,
the structure equations for a holomorphic curve $\Gamma$ in $\mathcal{X}(\Sigma)$ are 
given by \eq{pholoeqgen} with: $\bfu$ the immersion of $\Sigma$ in $\mathcal{S}^6$; $\bff$ a unitary frame for 
$T\mathcal{S}^6|_{\Sigma}$; $\theta=(\theta_1,\theta_2,0)$; and $\kappa$ satisfying $\kappa_{31}\w\kappa_{32}=0$.

\begin{remark}
One can use the formulae \eq{pholoframe} for a unitary frame for $T\mathcal{S}^6|_{\Sigma}$ to give a ``more explicit'' 
expression for $L_6(\bfu,\Sigma,\Gamma)$.  However, in doing so, one adapts frames so that 
$\bff_2$ and $\bff_3$ span $N_1\Sigma$ and $N_2\Sigma$.  This 
 breaks the symmetry of the problem and 
makes it difficult to see the holomorphic condition on $\Gamma$ in the formula 
one derives for the Lagrangian.  We therefore refrain from giving this expression for $L_6(\bfu,\Sigma,\Gamma)$.
\end{remark}

We now prove the main result in this paper.

\begin{thm}\label{mainthm}
A connected Lagrangian $L$ in $\mathcal{S}^6$ is ruled if and only if there exists an open dense subset 
$L^*$ of $L$ such that, for all $x\in L^*$, there exists an open set $U\ni x$ such that:
\begin{itemize}
\item[\emph{(a)}] $U\cap L^*=L_3(\bfu_1,\Sigma_1)$ for some holomorphic curve $\bfu_1:\Sigma_1\rightarrow\C\P^2$ 
as in Example \ref{s4ex3} and we may take $U=L^*=L$; or
\item[\emph{(b)}] $U\cap L^*=L_4(\bfu_2,\Sigma_2)$ for some 
 non-totally geodesic pseudoholomorphic curve $\bfu_2:\Sigma_2\rightarrow\mathcal{S}^6$ as in Example \ref{s4ex4}; or
\item[\emph{(c)}] $U\cap L^*=L_6(\bfu_3,\Sigma_3,\Gamma)$ for some non-totally geodesic pseudoholomorphic 
curve $\bfu_3:\Sigma_3\rightarrow\mathcal{S}^6$, and some holomorphic curve $\Gamma$ in $\mathcal{X}(\Sigma_3)$ as in Example \ref{s7ex1}.
\end{itemize}
Moreover, $L$ is not linearly full if and only if $L$ is of type (a) 
and the general family of ruled Lagrangians in $\mathcal{S}^6$ is locally classified by Example \ref{s7ex1}. 
\end{thm}

\begin{note}
Theorem \ref{mainthm} gives `\emph{Weierstrass formulae}' for ruled Lagrangians in $\mathcal{S}^6$ in the following sense: 
we can define them either using a holomorphic curve in $\C\P^2$ or, given a pseudoholomorphic curve in $\mathcal{S}^6$,
 using holomorphic data.  However, we should stress that the general pseudoholomorphic curve in $\mathcal{S}^6$ does \emph{not} 
admit a Weierstrass representation (though the null-torsion curves do by \cite[Theorem 4.7]{Bryant1}).  
Therefore, we really have a \emph{family} of Weierstrass representations for the 
general ruled Lagrangian, indexed by pseudoholomorphic curves in $\mathcal{S}^6$.  This is in stark contrast to 
ruled special Lagrangians in $\C^3$ which, it is believed, cannot be similarly described using Weierstrass formulae.  
\end{note}


\begin{proof} 
If $L$ has fundamental cubic $C_L$ with pointwise symmetry then by Proposition \ref{s4prop4} and the observations at the 
start of this section it must locally be given by (a) or (b) depending on whether it is linearly full or not.  It also follows 
from Proposition \ref{s4prop4} that $L$ is of type (a) if and only if $L$ is not linearly full 
and the local description can be extended to a global one.

Therefore, suppose $C_L$ does not have a pointwise symmetry.  Then, on some open dense subset $L^*$, the structure 
equations are given by \eq{ruledeqs1}-\eq{ruledeqs3}.  Recall the notation in $\S$\ref{pholo} and define:
\begin{subequations}\label{Xforms}
\begin{gather}
\bfu=-J\bfe_1;\\
\bff_1=\frac{1}{2}( -J\bfe_2+iJ\bfe_3);\qquad \bff_2=\frac{1}{2}(\bfe_2+i\bfe_3);\qquad \bff_3=\frac{1}{2}(\bfx+i\bfe_1);\\
\theta_1=\frac{1}{2}\big((t_0+2)+it_1\big)(\omega_2+i\omega_3);\qquad 
\theta_2=\frac{1}{2}\left(\frac{is}{4}\right)(\omega_2+i\omega_3);\\
\kappa_{22}=i\big((1+t_0)\omega_1+t_3\omega_2-t_2\omega_3\big);\qquad\kappa_{33}=-2i\omega_1;\label{kappa1}\\
\kappa_{31}=-\theta_2;\qquad\kappa_{21}=\frac{s}{4}\,\omega_1+\frac{a-ib}{4} (\omega_2+i\omega_3);\label{kappa2}\\
\kappa_{32}=\frac{1}{2}\big((t_0-2)+it_1\big)(\omega_2+i\omega_3).\label{kappa3}
\end{gather}
\end{subequations}
Let $\theta=(\theta_1,\theta_2,0)^{\rm T}$ and let $\kappa$ be the $3\times 3$ traceless skew-Hermitian matrix of 1-forms 
defined using \eq{kappa1}-\eq{kappa3}.  Using \eq{ruledeqs1}-\eq{ruledeqs3}, 
 we find that $\bfu$, $\bff=(\bff_1\;\bff_2\;\bff_3)$, $\theta$ and $\kappa$ satisfy the $\GGG_2$ 
structure equations \eq{pholoeqgen}. 

We observe that $\d\bfu$ is independent of $\omega_1$ and that left-multiplication by $\bfu$ acts as a complex structure map
on $\bff=(\bff_1\;\bff_2\;\bff_3)$ (this is easily verified by taking explicit imaginary octonionic representatives).  Thus, 
$\bfu$ defines a pseudoholomorphic curve $\bfu:\Sigma\rightarrow\mathcal{S}^6$ and $\bff$ is a unitary frame 
for $\bfu^*(T\mathcal{S}^6)$.   
  Moreover, $\kappa_{21}\neq 0$ since $s$ and $a^2+b^2$ are non-zero functions, so $\Sigma$ must be non-totally geodesic by 
the observations in $\S$\ref{pholo}.  Further, we see that $\kappa_{31}\w\kappa_{32}=0$.  

By the observations after Example \ref{s7ex1}, the structure equations satisfied by $\bfu$, $\bff$, 
$\theta$ and $\kappa$ define a holomorphic curve $\Gamma$ in $\mathcal{X}(\Sigma)$ with respect to the complex structure 
$I$ given by Theorem \ref{Xcomplex}.  Since $\bfx=\bff_3+\bar{\bff}_3$,
we deduce from Theorem \ref{Xcomplex} that $L^*$ is locally of the form $L_6(\bfu,\Sigma,\Gamma)$ as in Example 
\ref{s7ex1}.
\end{proof}

\begin{remark}
The general ruled Lagrangian admits two types of deformations: deformations of the underlying pseudoholomorphic curve $\Sigma$, 
and deformations of the holomorphic curve in $\mathcal{X}(\Sigma)$.
\end{remark}

To conclude this section, we study some reductions of the system for ruled Lagrangians.


\begin{ex}\textbf{(Ruled tubes in the first normal bundle)}\label{s7ex2}
Let $\bfu:\Sigma\rightarrow\mathcal{S}^6$ be a pseudoholomorphic curve with no totally geodesic points. 
Let $h_{\Sigma}$ be the second fundamental form of $\Sigma$ and let 
$\mathcal{U}(\Sigma)$ be its unit tangent bundle.  Define $\bfx:\mathcal{U}(\Sigma)\rightarrow\mathcal{S}^6$ by 
$$\bfx:\bft\mapsto \frac{h_{\Sigma}(\bft,\bft)}{\|h_{\Sigma}(\bft,\bft)\|}.$$
By \cite[Theorem 1]{Vrancken2}, $\bfx$ is a Lagrangian immersion in $\mathcal{S}^6$ if and only if $\Sigma$ is null-torsion.
 From Proposition \ref{tubes} and \eq{pholoframe}, we recognise these examples as ruled Lagrangian tubes 
in $N_1\Sigma$ about null-torsion pseudoholomorphic curves $\Sigma$.  
\end{ex}

The structure equations \eq{ruledeqs1}-\eq{ruledeqs3} for $t_1=0$ define 
an involutive EDS with $s_1=2$.  By comparing these structure equations with the calculations leading to the 
proof of \cite[Theorem 2]{Vrancken2}, we deduce that the solutions to this reduced system are given by Example \ref{s7ex2}.

\begin{remark}
By \cite[Theorem 2]{Vrancken2}, Lagrangians that admit a Killing vector whose integral curves are geodesic circles 
are classified by Example \ref{s4ex3}, $L_4(\bfu,\Sigma)$ as in Example \ref{s4ex4} for a null-torsion 
pseudoholomorphic curve $\Sigma$, and Example \ref{s7ex2}.
\end{remark}

\begin{ex}\textbf{(Products 2)}
The structure equations \eq{ruledeqs1}-\eq{ruledeqs3} with $t_0=0$ 
have $\d\omega_1=0$ and the reduced EDS is involutive with $s_1=2$.  Notice that $t_1$ must be non-zero by \eq{ruledeqs3}. 
 Thus, by Theorem \ref{mainthm}, the corresponding Lagrangians are locally \emph{products} $\mathcal{S}^1\times P$ for a
 surface $P$ in $\mathcal{S}^6$, and the Lagrangians are of the form $L_6(\bfu,\Sigma,\Gamma)$ 
as in Example \ref{s7ex1}.   For $L_6(\bfu,\Sigma,\Gamma)$ to be a product, the 
plane-bundle over $\Sigma$ defined by $\Gamma$ must be trivial.  
 We see from \eq{Xforms} that $t_0=0$ if and only if the torsion $k_1$ of $\Sigma$ satisfies $|k_1|=1$. Thus, 
$\Sigma$ is linearly full in a totally geodesic $\mathcal{S}^5$ in $\mathcal{S}^6$ and is minimal Legendrian by observations in 
$\S$\ref{pholo}.  Thus we get the same Lagrangians as in Example \ref{prodsex1}.
\end{ex}

\section[The quasi-ruled Lagrangian family]{The quasi-ruled Lagrangian family}\label{s8}

Here we prove that the fundamental cubic of a quasi-ruled Lagrangian has a pointwise symmetry, so 
they have already been classified by examples in $\S$\ref{s6}. 

\begin{thm}\label{mainthm2}
A connected Lagrangian $L$ in $\mathcal{S}^6$ is quasi-ruled if and only if its fundamental cubic has a pointwise 
$\SO(2)$, $\AAA_4$ or $\Z_3$-symmetry.  Therefore, $L$ has a $\frac{2}{3}$-ruling and there exists 
an open dense subset $L^*$ of $L$ such that, for all $x\in L^*$, there exists an open set $U\ni x$ such that 
either $U\cap L^*=L_1$ as given in Example \ref{s4ex2}, or $U\cap L^*=L_5(\bfu,\Sigma)$ for some non-totally 
geodesic null-torsion 
pseudoholomorphic curve $\bfu:\Sigma\rightarrow\mathcal{S}^6$ as in Example \ref{s4ex5}.
%
\end{thm}

\begin{proof}
Suppose that $L$ is a connected Lagrangian submanifold of $\mathcal{S}^6$ with a $\lambda$-ruling, for some $\lambda\in(0,1)$.  
By Lemma \ref{rcubic}, 
using the notation there, we can choose a coframe of $L$ such that its fundamental cubic is $C_L=C(r,s,a,b)$ for 
$r=\frac{4}{\lambda}\,\sqrt{1-\lambda^2}\neq 0$ and some functions $s,a,b$.  
  If $C_L$ has a pointwise symmetry we are done by Tables \ref{symtable} and \ref{summary}, so, for a contradiction, we assume otherwise. 
%

If $s=0$, we can choose $b=0$ by Lemma \ref{rcubic}, but $C(r,0,a,0)$ has a pointwise symmetry  by Table \ref{symtable}.  
We also notice that $C(r,s,0,0)$ has a pointwise symmetry, so 
$s\neq 0$ and $a^2+b^2\neq0$ 
on some open dense subset $L^*$ of $L$.  

As $C_L$ has a trivial stabilizer in $\SO(3)$ at each point, there exist functions $t_{ij}$ on $L^*$ 
such that $\alpha_i=t_{ij}\omega_j$.  
Since $L$ is quasi-ruled in the $\bfe_1$ direction, $t_{21}=t_{31}=0$
.  From the structure equations \eq{secondLa}, \eq{secondLb} 
and \eq{secondLd} we find that:
\begin{subequations}\label{quasi}
\begin{gather}
s\left(2t_{11}+1\right)=-(2r+s)\left(2t_{22}+1\right)=(2r-s)\left(2t_{33}+1\right)\quad\text{and}\\
(2r+s)t_{23}=(2r-s)t_{32}.
\end{gather}
\end{subequations}
Therefore, we can split our discussion into two cases: $|s|=2r$ and $|s|\neq 2r$.

Suppose $|s|=2r$ and further, without loss of generality, that $s=-2r$. Equations \eq{quasi} imply that
$t_{11}=t_{33}=-\frac{1}{2}$ and $t_{32}=0$.  We also find, since $s$ is constant, that 
\begin{equation}\label{quasi3}
8rt_{12}=-a\left(2t_{22}+1\right)-2bt_{23}\quad\text{and}\quad 8rt_{13}=-2at_{23}+b\left(2t_{22}+1\right).
\end{equation}
Putting this information in \eq{secondLc} leads to $r^2=4$, so $r=2$, and 
\begin{equation}\label{quasi4}
2t_{12}t_{23}-t_{13}(2t_{22}+1)-b=0.
\end{equation}
From \eq{quasi3}-\eq{quasi4}, we deduce that $b=0$ and that $t_{12}$ and $t_{13}$ are multiples, 
depending on $a$, of $2t_{22}+1$ and $t_{13}$ respectively. Further, we see that $a$ is necessarily constant and, since
$$\d a= 3\left(1-\textstyle\frac{1}{16}a^2\right)\left(2t_{23}\omega_2-(2t_{22}+1)\omega_3\right),$$
we must have that $a^2=16$ or $2t_{22}+1=t_{23}=0$.  The former case quickly leads to a contradiction from \eq{secondLd}, whereas the latter forces $a=0$, contradicting our assumption that $a^2+b^2\neq 0$.  

Hence we turn our attention to the possibility of $|s|\neq 2r$ on $L^*$.  By \eq{quasi}, 
there exist functions $t_0$ and $t_1$ on $L^*$ such that:
\begin{subequations}\label{quasi5}
\begin{gather*}
2t_{11}+1=4(4r^2-s^2)t_0;\quad 2t_{22}+1=-4s(2r-s)t_0;\quad  2t_{33}+1=4s(2r+s)t_0;\\
t_{23}=2s(2r-s)t_1;\quad\text{and}\quad t_{32}=2s(2r+s)t_1.
\end{gather*}
\end{subequations}
One quickly finds from \eq{secondLc} that 
$s$ satisfies $r^2+s^2=20$, so $s$ is constant.  
This then forces $t_1=0$, $t_{12}=-4rat_0$ and $t_{13}=4rbt_0$.  Putting this information back in the structure equations \eq{secondLc}-\eq{secondLd} leads to $s=b=0$.  This is again a contradiction and the theorem is proved.
\end{proof}

We conclude with the analogue of Proposition \ref{ruledsimple}.

\begin{prop}
A connected quasi-ruled Lagrangian has a unique $\frac{2}{3}$-ruling.
\end{prop}

\begin{proof}
By Theorem \ref{mainthm2}, a connected quasi-ruled Lagrangian $L$ has a $\frac{2}{3}$-ruling.  Moreover, the fundamental 
cubic $C_L$ has $\SO(2)$, $\AAA_4$ or $\Z_3$-symmetry.  Suppose a unit tangent vector $\bfe$ on $L$ is in the direction 
of a $\frac{2}{3}$-ruling. Then $h_L(\bfe,\bfe)=2\sqrt{5} J\bfe$, 
where $h_L$ is the second fundamental form of $L$.  
By the work in $\S$\ref{s6}, using the notation of Lemma \ref{rcubic}, $C_L=C(2\sqrt{5},0,a,0)$ for some non-negative function $a$, 
up to $\SO(3)$ transformation.  
Elementary algebra then shows that there is a unique vector $\bfe$
in each case, namely $\bfe=\bfe_1$ in this frame. 
\end{proof}




\begin{ack} 
The author thanks Robert Bryant and Daniel Fox for useful conversations, and MSRI for hospitality during the 
time of this research.  The author was supported by an NSF Mathematical Sciences Postdoctoral Research Fellowship.
\end{ack}

\medskip

\noindent \textsl{Address for correspondence:\\[4pt] Imperial College\\ London SW7 2AZ\\ England, U.K.\\
j.lotay@imperial.ac.uk}

\end{document}